\documentclass[12pt]{amsart}
\usepackage{graphicx}
\usepackage{mathrsfs}
\usepackage{epsfig}
\usepackage{amsmath}
\usepackage{amsfonts}
\usepackage{amssymb}
\usepackage{amssymb,amscd}
\usepackage{tikz}
\usepackage{verbatim}
\usepackage{booktabs}
\usepackage[all]{xy}

\newtheorem{thm}{Theorem}[section]

\newtheorem{prop}[thm]{Proposition}

\theoremstyle{definition}
\newtheorem{defn}[thm]{Definition}
\theoremstyle{remark}

\numberwithin{equation}{section}
\newcommand{\arc}[2]{{#1}\curvearrowright{#2}}

\newcommand{\la}{\langle}
\newcommand{\ra}{\rangle}
\newcommand{\HdC}{\mathbf H^{2}_{\mathbb C}}

\newcommand{\Cdu}{\mathbb C^{2,1}}
\newcommand{\Ct}{\mathbb C^{3}}
\newcommand{\C}{\mathbb C}
\newcommand{\Cp}{\mathbb{CP}^2}
\newcommand{\p}{\mathbb P}
\newcommand{\R}{\mathbb R}

\newcommand{\Pu}{\mbox{\rm PU(2,1)}}
\newcommand{\Su}{\mbox{\rm SU(2,1)}}

\newcommand{\B}{\mathcal{B}}

\newcommand{\bz}{{\bf z}}

\newcommand{\Z}{{\mathbb Z}}

\renewcommand{\Im}{\mbox{\rm Im}\,}
\renewcommand{\leq}{\leqslant}
\renewcommand{\geq}{\geqslant}

\begin{document}
	
	
	\date{Oct. 4, 2023}

	\title[]{Figure-eight knot is always over there}
	\author[J. Ma]{Jiming Ma}
	\address{School of Mathematical Sciences, Fudan University, Shanghai, China}
	\email{majiming@fudan.edu.cn}
	
	\author[B. Xie]{Baohua Xie}
	\address{School of Mathematics, Hunan University, Changsha, China}
	\email{xiexbh@hnu.edu.cn}
	
	
	
	\thanks{Jiming Ma was  supported by NSFC (No.12171092), he is also a member of LMNS, Fudan University. Baohua Xie was supported by NSFC (No.11871202).}
	
    \begin{abstract}
It is well-known that complex hyperbolic triangle groups $\Delta(3,3,4)$  generated by three  complex  reflections $I_1,I_2,I_3$ in $\Pu$  has 1-dimensional moduli space. Deforming the representations from the classical $\mathbb{R}$-Fuchsian one to $\Delta(3,3,4; \infty)$, that is, when $I_3I_2I_1I_2$ is accidental parabolic, the 3-manifolds at infinity change from a Seifert 3-manifold to the  figure-eight knot complement.

 When  $I_3I_2I_1I_2$ is loxodromic, there is an open set  $\Omega \subset \partial\HdC=\mathbb S^3$ associated to   $I_3I_2I_1I_2$, which is a subset  of  the  discontinuous region. 	We show  the quotient space   $\Omega/ \Delta(3,3,4)$ is always the  figure-eight knot complement in the  deformation process. This gives the  topological/geometrical explain  that the 3-manifold at infinity of $\Delta(3,3,4; \infty)$  is the   figure-eight knot complement. In particular, this confirms  a  conjecture  of Falbel-Guilloux-Will.

    \end{abstract}

     \maketitle

	\section{Introduction}
		Let ${\bf H}^2_{\mathbb C}$ be the complex hyperbolic plane,   the   holomorphic isometry group of ${\bf H}^2_{\mathbb C}$ is $\Pu$. 	  The complex hyperbolic plane ${\bf H}^2_{\mathbb C}$ can be identified with the unit ball in ${\mathbb C}^2$, so the ideal boundary  $\partial {\bf H}^2_{\mathbb C}$ of  ${\bf H}^2_{\mathbb C}$ is  the 3-sphere $\mathbb{S}^3$. 
		
		Thurston's work on 3-manifolds has shown that geometry has an important role  in the study of topology of 3-manifolds.
	We have  three kinds of geometrical  structures on 3-manifolds related to  the  pair $({\bf H}^2_{\mathbb C}, \partial {\bf H}^2_{\mathbb C})$  with increasing group action constraints.
	
	\begin{defn} \label{def:SCR}For a smooth 3-manifold $M$:
	\begin{enumerate} 
		\item \label{SCR}	A {\it spherical CR-structure} on  $M$ is a maximal collection of distinguished charts modeled on the boundary $\partial \mathbf{H}^2_{\mathbb C}$, where coordinates changes are restrictions of transformations from  $\Pu$.
		In other words, a {\it spherical CR-structure} is a $(G,X)$-structure with $G=\Pu$ and $X=\mathbb{S}^3$;
		\item \label{CRSU}  On the other hand, a  {\it CR-structure  spherical  uniformization}  of $M$ is a homeomorphism $M = \Omega/  \rho(\pi_{1}(M))$, where $\Omega$  is an open subset of $\partial\HdC$ on which $  \rho(\pi_{1}(M))$ acts properly  discontinuously. See \cite{Falbel-Guilloux-Will, Kas18};
		\item  \label{USCR} A spherical CR-structure on $M$  is {\it uniformizable} if it is
		obtained as $M= \Omega_{\Gamma}/\Gamma$, where  $\Omega_{\Gamma}\subset \partial \mathbf{H}^2_{\mathbb C}$ is the  discontinuity region  of a discrete subgroup  $\Gamma$. The \emph{limit set} $\Lambda_{\Gamma}$ of $\Gamma$ is  $\partial \mathbf{H}^2_{\mathbb C}- \Omega_{\Gamma}$ by definition.
	\end{enumerate}
	\end{defn}
	For a discrete group $\Gamma< \Pu$, the open set  $\Omega$ in (\ref{CRSU}) of Definition \ref{def:SCR} is a subset of the  discontinuity region $\Omega_{\Gamma}$  in (\ref{USCR})  of Definition \ref{def:SCR}. So for a discrete group $\Gamma$, there is at most one  uniformizable spherical CR-structure associated to it, but  there may be infinitely many CR-structure  spherical  uniformizations associated to it.

	 For a discrete group $\Gamma < \Pu$, the 3-manifold $M=\Omega_{\Gamma}/\Gamma$ at infinity of the 4-manifold $\mathbf{H}^2_{\mathbb C}/ \Gamma$  is the analogy of the 2-manifold at infinity of a geometrically finite,  infinite volume hyperbolic 3-manifold. In other words, uniformizable spherical CR-structures  on 3-manifolds in $\mathbf{H}^2_{\mathbb C}$-geometry are the analogies
	of  conformal structures on surfaces in $\mathbf{H}^3_{\mathbb R}$-geometry.
	
	In  the three kinds of  geometrical structures of Definition \ref{def:SCR},   uniformizable spherical CR-structures  on 3-manifolds seem to be the most interesting ones. But in contrast to results on other geometric structures carried on 3-manifolds, there are relatively few examples known about them. 	 A possible way to get   uniformizable spherical CR-structures is via the deformations of triangle groups  in $\Pu$.

	Let $T(p,q,r)$ be the abstract  triangle group with presentation
	$$T(p,q,r)=\langle \sigma_1, \sigma_2, \sigma_3 | \sigma^2_1=\sigma^2_2=\sigma^2_3=(\sigma_2 \sigma_3)^p=(\sigma_3 \sigma_1)^q=(\sigma_1 \sigma_2)^r=id \rangle,$$
	where $p,q,r$ are positive integers or $\infty$ satisfying $$\frac{1}{p}+\frac{1}{q}+\frac{1}{r}<1.$$  We assume that $p \leq q \leq r$. If $p, q$ or $r$ equals $\infty$, then
	the corresponding relation does not appear.  The  ideal triangle group is the case that $p=q=r=\infty$.
	A \emph{$(p,q,r)$ complex hyperbolic triangle group} is a representation $\rho$ of $T(p,q,r)$ into $\Pu$
	where the generators fix complex lines. We denote $\rho(\sigma_{i})$ by $I_{i}$, and the image group by $\Delta(p,q,r)=\langle I_1,I_2,I_3 \rangle$. It is well known  \cite{schwartz-icm} that the space of $(p,q,r)$-complex reflection triangle groups has real dimension one if $3 \leq p \leq q \leq r$.

	The  isometry group of the real hyperbolic plane ${\bf H}^2_{\mathbb R}$ is $\mbox{\rm PO(2,1)}$, and  it is well known that the  ideal triangle group is rigid in  $\mbox{\rm PO(2,1)}$.  Goldman and Parker  \cite{GoPa} initiated the study of the deformations of ideal triangle group into  $\Pu$. 
	They  gave an interval  in the moduli space of complex hyperbolic ideal triangle groups, for points in this interval  the corresponding representations are discrete and faithful.
	They conjectured that a complex hyperbolic ideal triangle group $\Delta(\infty,\infty, \infty)=\langle I_1, I_2, I_3 \rangle$ is discrete and faithful if and only if $I_1 I_2 I_3$ is not elliptic. Schwartz proved  Goldman-Parker's conjecture in \cite{schwartz-go-p1, schwartz-go-p2}. Furthermore,  Schwartz analyzed the complex hyperbolic ideal triangle group $\Gamma$ when $I_1 I_2 I_3$ is parabolic, and showed  the 3-manifold at infinity of the quotient space ${\bf H}^2_{\mathbb C}/{\Gamma}$ is commensurable with the
	Whitehead link complement in the 3-sphere \cite{schwartz-litg}. In other words, the Whitehead link complement admits uniformizable spherical CR-structure. Seifert 3-manifolds admitting  uniformizable spherical CR-structures are rather easy to construct, but the  Whitehead link complement  is the first example of  hyperbolic 3-manifold which admits  uniformizable spherical CR-structure.

	 Richard Schwartz \cite{schwartz-icm} has also  conjectured the necessary and sufficient condition for a general complex hyperbolic  triangle group $$\Delta(p,q,r)=\langle I_1,I_2,I_3\rangle < \Pu$$ to be a discrete and faithful  representation of $T(p,q,r)$. Schwartz's conjecture has been proved in a few cases \cite{jwx, pwx,par-will2}.


 The \emph{critical point} of the 1-dimensional deformation space of complex hyperbolic  triangle groups is a point such that    some preferred word $W_{A}$ or $W_{B}$ is accidental parabolic. For more details,  see  \cite{schwartz-icm}. People found   several  more examples of cusped hyperbolic 3-manifolds which admit uniformizable spherical CR-structures  at these critical points \cite{acosta, deraux-scr, der-fal, jwx,  mx, MaXie2021}. 	Almost all of the   examples of  uniformizable spherical CR-structures   gotten now  are  via difficult and  sophisticated analysis. But we do not know the topological/geometrical reason  the 3-manifolds at infinity of the groups associated to  critical points should be the ones we got. Falbel-Guilloux-Will \cite{Falbel-Guilloux-Will} proposed a method to predict the 3-manifold when there is an accidental parabolic element. 
	
	We now just consider the representations of $T(3,3,4)$ into   $\Pu$ with complex reflection generators $I_1,I_2,I_3$. We can parametrize the representations by $t \in [1/3, \sqrt{2}-1]$, and the even subgroup of  the image group is denoted by $\Gamma_{t}$, see Section \ref{section:334representation} for more details. Moreover,
	\begin{itemize} 
		\item When  $t= \sqrt{2}-1$, the image group lies in  $\mbox{\rm PO(2,1)}$. So we have the classical   $\mathbb{R}$-Fuchsian group;
		\item For any  $t \in (3/8, \sqrt{2}-1]$,  $I_3I_2I_1I_2$ is loxodromic;
		\item When  $t=3/8$, $I_3I_2I_1I_2$ is parabolic. This is an accidental parabolicity, so  $t=3/8$ corresponds to the  critical point in the moduli space of $\Delta(3,3,4)$ in our parameterization;  
		\item When  $t \in [1/3, 3/8)$, $I_3I_2I_1I_2$ is elliptic. We will not consider  representations in this interval.
	\end{itemize}

It is showed by Parker-Wang-Xie  \cite{pwx} for each  $t \in [3/8, \sqrt{2}-1]$, the corresponding representation is discrete and faithful. Since   when  $t=\sqrt{2}-1$,   we have a $\mathbb{R}$-Fuchsian group, so 3-manifold at infinity of the corresponding group is just  the unit tangent bundle over the real hyperbolic $(3,3,4)$-orbisurface.  But  when   $t=3/8$, there is a new  parabolic element $I_3I_2I_1I_2$, so the 3-manifold at infinity of the corresponding group must change. It is showed by Deraux-Falbel \cite{der-fal}  the 3-manifold at infinity of the even subgroup   $\Gamma_{3/8}$ is the figure-eight knot complement. But we do not know the reason that the 3-manifold at infinity of
	  $\Delta(3,3,4)$ when   $I_3I_2I_1I_2$ is  parabolic  should be this one. 
 Falbel-Guilloux-Will \cite{Falbel-Guilloux-Will} proposed  an explanation of this phenomenon.
 
For all  $t \in (3/8, \sqrt{2}-1]$,  $I_3I_2I_1I_2$ is loxodromic. Let $p_1$ and $p_2$ be the attractive and repulsive fixed points of it, they determine a $\C$-circle. We denote by $\alpha_1$ a preferred one of the two arcs with end points  $p_1$ and $p_2$  in the  $\C$-circle (see Section \ref{section:Rfuchsian} for this arc).  Let $\Lambda_{t} $ be the limit set of  $\Gamma_{t}$. Then it is a topological circle. The \emph{crown} associated to $I_3I_2I_1I_2$ is the subset of $\mathbb{S}^3$ defined as 
 $${\rm Crown}={\rm Crown}_{\Gamma_{t},I_3I_2I_1I_2}= \Lambda_{t} \cup\Bigl(\bigcup_{g\in\Gamma_{t}} g\cdot\alpha_1\Bigr).$$
 We denote $\Omega_{\Gamma_{t},I_3I_2I_1I_2}\subset \Omega_{\Gamma_{t}}$ as the complement of ${\rm Crown}_{\Gamma_{t},I_3I_2I_1I_2}$ in $\mathbb{S}^3$. Recall that $\Omega_{\Gamma_{t}}=\mathbb{S}^3-\Lambda_{t}$ is the discontinuous region of $\Gamma_{t}$'s action on $\mathbb{S}^3$.

  It was  shown in \cite{Dehornoy} that $\Omega_{\Gamma_{t},I_3I_2I_1I_2}/\Gamma_{t}$ is homeomorphism to  the  figure-eight knot complement when $t=\sqrt{2}-1$. In fact  Falbel-Guilloux-Will \cite{Falbel-Guilloux-Will} identified  this manifold as drilling out  the unit tangent bundle of $(3,3,4)$-orbisurface a certain closed orbit associated to $I_3I_2I_1I_2$. Moreover, Falbel-Guilloux-Will \cite{Falbel-Guilloux-Will} conjectured that the quotient space of $\Omega_{\Gamma_{t},I_3I_2I_1I_2}$ by $\Gamma_{t}$ is always the  figure-eight knot complement for any $t \in (3/8, \sqrt{2}-1)$. So each of them gives a  CR-structure  spherical  uniformization of   figure-eight knot complement.  The last one, that is when $t=3/8$, gives the uniformizable spherical CR-structure on the figure-eight knot complement. Which corresponds to pinching on the limit set of  $\Gamma_{t}$ to the limit set of  $\Gamma_{3/8}$. So this conjecture explains how to get  the 3-manifold at infinity of  $\Gamma_{3/8}$ from the  3-manifold at infinity of  a  $\mathbb{R}$-Fuchsian group.
   Falbel-Guilloux-Will \cite{Falbel-Guilloux-Will} confirmed the conjecture when $t$ is near to $\sqrt{2}-1$.

	We certificate Falbel-Guilloux-Will's conjecture totally   in this paper:
	
	\begin{thm} \label{thm:figure8} For the parameterazation of complex hyperbolic groups $\Delta(3,3,4)$ by $t \in (3/8, \sqrt{2}-1]$:
		\begin{enumerate} 
			\item \label{unittangentbundle} The 3-manifold at infinity of the even subgroup  $\Gamma_{t}$ is the unit tangent bundle of the $(3,3,4)$-orbisurface for all $t \in (3/8, \sqrt{2}-1]$;
			\item \label{figure8} The quotient space of $\Omega_{\Gamma_{t},I_3I_2I_1I_2}$ by $\Gamma_{t}$ is always the  figure-eight knot complement  for all $t \in (3/8, \sqrt{2}-1]$.
			\end{enumerate} 
	\end{thm}

So in the  deformation process, the figure-eight knot is always over there! This explains the 3-manifold at infinity of the even subgroup  $\Gamma_{3/8}$ (with accidental parabolic element) is  the figure-eight knot \cite{der-fal}.

We prove Theorem \ref{thm:figure8} in the following steps:
	\begin{itemize}
		\item  For  $\Delta(3,3,4)=\langle I_1,I_2,I_3\rangle$ depends on $t \in (3/8, \sqrt{2}-1]$,   $I_1I_2$ has order 4, and $I_1I_2$ has fixed point $p_0 \in \mathbf{H}^2_{\mathbb C}$;
			\item  Consider  the Dirichlet domain $D_{t}$ of $\Gamma_{t} < \Delta(3,3,4)$ with center $p_0$, $D_{t}$ has eight facets \cite{pwx};
			\item The ideal boundary  $\partial_{\infty}D_t=D_{t}\cap \partial \mathbf{H}^2_{\mathbb C}$ is a solid torus. Moreover, the  boundary  of $\partial_{\infty}D_t$  consists of eight annuli, the  side-pairing pattern on them is independent of $t \in (3/8,  \sqrt{2}-1]$. So  the 3-manifold at infinity of the  group  $\Gamma_{t}$ is independent of $t$. This proves (\ref{unittangentbundle}) of Theorem \ref{thm:figure8};
			\item We then consider the complement  of the crown  in  $\partial_{\infty}D_{t}$, that is, $$ \partial_{\infty}D_{t}-{\rm Crown}_{\Gamma_{t},I_3I_2I_1I_2}.$$ 
			Which is a fundamental domain of  $\Gamma_{t}$'s action on $\Omega_{\Gamma_{t},I_3I_2I_1I_2}$. 
			In fact $\partial_{\infty}D_{t}\cap {\rm Crown}_{\Gamma_{t},I_3I_2I_1I_2}$ are exactly eight arcs. 
		We will show the topology and 	the  side-pairing pattern on $\partial_{\infty}D_{t}-{\rm Crown}_{\Gamma_{t},I_3I_2I_1I_2}$ are independent of $t$. This in turn  proves (\ref{figure8}) of Theorem \ref{thm:figure8}.
		\end{itemize}

	\textbf{Acknowledgement}: Part of the work was carried out when Jiming Ma was  visiting  Hunan University in the summer of 2022, the hospitality  is gratefully appreciated. The second author thanks John Parker for a useful discussion about the parametrization  of the deformation space of the triangle group $\Delta(3,3,4)$ several years ago.

\section{Background}\label{sec:background}
 We will briefly introduce some background of complex hyperbolic geometry  in this section. One can refer to Goldman's book \cite{Go} for more details.

\subsection{Complex projective space and complex hyperbolic plane}
  The projective space $\Cp$ is the quotient of the complex space $\Ct$ minus the origin, by the non-zero complex numbers.  We denote by
$\p$ the projectivisation map $\p:\Ct\backslash\{0\}\to \Cp$.  We will constantly use points in the projective space $\Cp$ and lifts to
 $\Ct$(or in $\Cdu$, see below) throughout this paper.  In this situation, points in $\Ct$ will be denoted by $\bz$, and $z$ will denote the 
 image in $\Cp$ under projectivisation.
 
 Let $\Cdu$ denote a copy of $\Ct$ equipped with a Hermitian form $\la\cdot ,\cdot \ra$ of signature $(2,1)$ on $\Ct$, and define
 \begin{align*}
 	V_{-}&=\{ Z\in \Ct: \la Z,Z \ra<0\},\\
 	V_{+}&=\{ Z\in \Ct: \la Z,Z \ra>0\},\\
 	V_{0}&=\{ Z\in \Ct: \la Z,Z \ra=0\}.
 \end{align*}
The complex hyperbolic plane $\HdC$ is the projectivsation of the cone $V_{-}$ in $\Cdu$, equipped with a Hermitian  metric
induced by the Hermitian form $\la\cdot ,\cdot \ra$. The projection  to $\Cp$ of the quadratic  $V_{0}$ can be thought of as 
the boundary at infinity of $\Cdu$, and we will denote it as $\partial\HdC$. The space $\HdC$ is homeomorphic to  a ball $B^4$, and $\partial\HdC$  is homeomorphic to 3-sphere $\mathbb S^3$.

The complex hyperbolic distance on $\HdC$ is given by 
$$\cosh\left(\frac{d(p,q)}{2}\right)= \frac{\left|\la {\bf p} ,{\bf q} \ra\right|^2  }{\left|\la{\bf p} ,{\bf p} \ra\right|\left|\la{\bf q} ,{\bf q} \ra\right| }.$$
The subgroup of $\mbox{\rm SL(3,$\C$)}$ of maps that preserve the Hermitian form $\la\cdot ,\cdot \ra$ is by definition $\Su$ and its projectivisation $\Pu$ the group of holomorphic 
isometries of $\HdC$. We will often work with 
$\Su$, which is a 3-fold cover of $\Pu$.

\subsection{Two models}
There are two special choices of the Hermitian forms
\begin{equation*}
	J_1=\left[
	\begin{array}{ccc}
		1 & 0 & 0 \\
		0 & 1 & 0 \\
		0 & 0 & -1 \\
	\end{array}
	\right]\quad {\rm and}\quad J_2=\left[
	\begin{array}{ccc}
		0 & 0 & 1 \\
		0 & 1 & 0 \\
		1 & 0 & 0 \\
	\end{array}\right].
\end{equation*}  
Note that they are conjugate by the Cayley  transformation
\begin{equation*}\label{eq:cayley}
	Cay=\frac{1}{\sqrt{2}}\left[
	\begin{array}{ccc}
		1 & 0 &1 \\
		0 & \sqrt{2} & 0 \\
		1 & 0 & -1 \\
	\end{array}
	\right].
\end{equation*}

By using the  Hermitian form  given by $J_1$, we
obtain the ball model of $\HdC$.  With this model, $\HdC$ can be seen as the unit ball in $\C^2$, where $\C^2$ itself is seen as the affine
chart $z_3=1$ of $\Cp$. Any point in $\HdC$ can be lifted to $\C^3$ in a  unique way as a vector
$[z_1,z_2,1]^{T}$, where $z_i\in\C$ and $\left| z_1\right|^2+\left| z_2\right|^2<1$.  The boundary  $\partial\HdC$ is just the 3-sphere $\mathbb S^3$ defined by   $\left|z_1\right|^2+\left| z_2\right|^2=1$.

The second model that one will consider is the Siegel model if one uses the form $J_2$. It will be more convenient to  analyze Heisenberg geometry and  draw pictures. In this model, the projection of $V_{-}\cup V_{0}$ to $\Cp$ is contained in the affine chart $z_3=1$, except for the 
projection of $[1,0,0]^T$, which is at infinity. Thus any point in the closure of $\HdC$ admits a unique lift to $\C^3$, which is given by
\begin{equation*}
\psi(z,t,u)=\left[
\begin{array}{c}
	\frac{-|z|^2-u+it}{2} \\
	z \\
	1 \\
\end{array}
\right]\quad {\rm and}\quad \left[
\begin{array}{c}
1 \\
0 \\
0 \\
\end{array}
\right],
\end{equation*}
where $z\in \C, t\in \R$ and $u\geq 0$. There coordinates are often called horospherical coordinates.
When necessary, we will  call the vector given above the
standard lift of a point in $\HdC$. We  will denote by 
$[z,t]$ the point in $\partial\HdC$ which is the projection of $\psi(z,t,0)$. Then one can identify $\partial\HdC$ with $\C\times\R\cup \{\infty\}\}$. Removing the point at infinity, we obtain the Heisenberg group, defined as  $\C\times\R$ with multiplication
\begin{equation*}
	\left(w,s\right)\ast \left(z,t\right)=\left(w+z,s+t+2\Im(w\bar{z}) \right). 
\end{equation*}

\subsection{Two totally geodesic submanifolds and their boundarys}\label{subsection:affinedisk}

There are two kinds of totally geodesic submanifolds of real dimension 2 in 
$\HdC$: \emph{complex lines} in $\HdC$ are complex geodesics(represented by $\mathbf{H^{1}_{\mathbb C}}$) and \emph{Langrangian planes} in $\HdC$ are totally
real geodesic 2-planes(represented by $\mathbf{H^{2}_{\mathbb R}}$). Each
of these  totally  geodesic submanifolds is a model of the hyperbolic plane.
A \emph{polar vector} of a complex line is the unique vector(up to scaling) in $V_{+}$ perpendicular to this complex line.

A discrete subgroup of  $\Pu$ preserving a complex line is called $\C$-Fuchsian and is isomorphic to a subgroup of $P(U(1)\times U(1,1))\subset \Pu$. A discrete subgroup of  $\Pu$ preserving a Langrangian plane is called $\R$-Fuchsian and is isomorphic to a subgroup of $SO(2,1)\in SU(2,1)$.

Consider the complex hyperbolic space  $\HdC$ and its boundary $\partial\HdC$.
We define the $\C$-circle in $\partial\HdC$ to be the boundary of a complex geodesic in  $\HdC$. Analogously, we define the $\R$-circle in $\partial\HdC$ to be the  boundary of a Langrangian plane  in $\HdC$.

\begin{defn} For a given complex geodesic $C$, a complex reflection with minor $C$ is the isometry $\iota_{C}$ in $\Pu$ given by 
	$$\iota_{C}=-z+2\frac{\la z,c\ra}{\la c, c\ra}c,$$ where $c$ is a polar vector of $C$.
\end{defn}

\begin{defn} The contact plane at $M=(a,b,c)$ is the plane
	$P(M):=Z-c+2aY-2bX$.
\end{defn}

The $\C$-circle of center $M=(a,b,c)$ and radius $R$ is the  intersection of
the contact plane at $M$ and the cylinder $(X-a)^2+(Y-b)^2=R^2$.

\begin{prop}
In the Heisenberg group, $\C$-circles are either vertical lines or ellipses
whose projections  on the $z$-plane are circles.
\end{prop}

For a given pair of distinct points in $\partial\HdC$, there is a unique
$\C$-circle passing through them. A finite $\C$-circle is determined by a center and a radius. For  example, the finite  $\C$-circle with center
$(z_0,t_0)$ and radius $R>0$ has a polar vector
\begin{equation*}
 \left[
	\begin{array}{c}
		(R^2-\left|z_0\right|^2+it_0)/2\\
		z_0 \\
		1 \\
	\end{array}
	\right],
\end{equation*}
and in it any point $(z,t)$ satisfies the equations
$$ \begin{cases}
\left| z-z_0\right|=R,\\
	t=t_0+2\Im(\bar{z}z_0). \\
\end{cases}$$

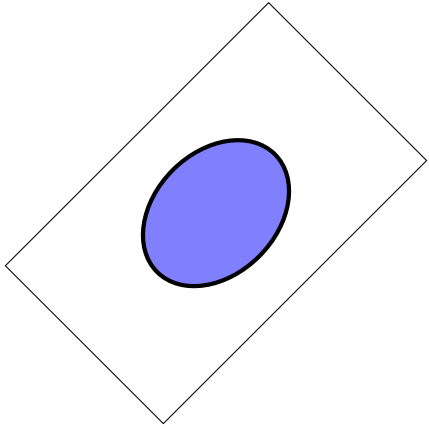
\begin{figure}
	\begin{center}
		\begin{tikzpicture}[scale=0.7] 
		\draw (-4,-1)--(1,4)--(4,1)--(-1,-4)--(-4,-1);	
		 \filldraw[rotate=45][fill = blue!50!white,draw = black][ultra thick] (0,0)
		 ellipse (45pt and 33pt);	
		\end{tikzpicture}	 
	\end{center}
	\caption{The affine disk bounded by the finite $\C$-circle in the contact plane.}
	\label{fig:cdisk}
\end{figure}

\begin{defn} \label{definition:cdisk}We define the  $\C$-disk to be the affine disk bounded by the finite $\C$-circle in the contact plane, see Figure \ref{fig:cdisk}.
\end{defn}
The condition for self-intersection between the complex lines defined by 
polar vectors $v_1$ and $v_2$ is
\begin{equation}\label{eq:linked-circle}
L(v_1,v_2)=\left|\la v_1, v_2\ra\right|^2-\la v_1, v_1\ra \la v_2, v_2\ra<0.
\end{equation}
This condition $L(v_1,v_2)>0$ was also known as a non-linking condition
for two $\C$-circles with polar vectors $v_1$ and $v_2$, see \cite{par}.

\subsection{Bisectors and Dirichlet domain}
There are no totally geodesic real hypersurface $\HdC$, and so we must choose hypersurfaces 
for sides  of our polyhedron. We choose to work with bisector. A bisector in $\HdC$ is the locus
of points equidistant (with respect to the Bergman metric) from a given pair of points in $\HdC$.
Suppose that these points are $u$ and $v$. Choose lifts $\mathbf{u}$, ${\mathbf{v}}$ of $u$ and $v$ so that
$\la \mathbf{u}, \mathbf{u}\ra =\la  \mathbf{v}, \mathbf{v}\ra$.  Then the bisector equidistant from
$u$ and $v$ is
$$\B=\B(u,v)=\{p\in \HdC: |\la  \mathbf{p}, \mathbf{u}\ra|=|\la  \mathbf{p}, \mathbf{v}\ra| \}.$$

Suppose that we are given three points $u, v_1$ and $v_2$ in $\HdC$.  If the three corresponding
vectors $u, v_1$ and $v_2$ in $V_{-}$ form a basis for  $\C^{2,1}$ then the intersection 
$\B(u,v_1)\cap \B(u,v_2)$ is called a Giraud disc.  This is a particularly nice type of bisector intersection.

Suppose that $\Gamma$  is a discrete group of $\Pu$.
Let $p_0$ be a point of $\HdC$ and write $\Gamma_{p_0}$ for 
the stabilizer of $p_0$ in $\Gamma$.  Then  the Dirichlet
domain $D_{p_0}(\Gamma)$ for  $\Gamma$  with centre 
$p_0$ is  defined to be 

$$D_{p_0}(\Gamma)=\{p\in \HdC: d(p,p_0)<d(p, g(p_0))\,\ {\rm for\,\ all} \,\ g\in \Gamma-\Gamma_{p_0}\}$$

We define the spinal sphere $\mathcal{S}\in  \partial\HdC$ as the boundary of the bisector $\B$ in $\HdC$.  Note that two  spinal spheres have an intersection if and only if the corresponding  bisectors have an intersection.

 \section{Complex hyperbolic triangle groups $\Delta(3,3,4)$}  \label{section:334representation}
 
 Let $I_i$ be a reflection along the complex line $C_i$ for $i=1,2,3$. We assume that $C_{i-1}$ and $C_{i}$
either meet at the angle $\pi/p_i$ for some integer $p_i\geq 3$ or else $C_{i-1}$ and $C_{i}$ are asymptotic, in which case they make an angle 0 and  we write $p_i=\infty$, where the indices are taken mod 3. The subgroup $\Delta(p_1,p_2,p_3)$ of $\Pu$
 generated by $I_1, I_2$  and $I_3$  is called a {\it complex hyperbolic triangle group}. For fixed $p_1,p_2,p_3$, modulo conjugacy in $\Pu$, there exists in general a $1$-parameter family of complex hyperbolic triangle group  $\Delta(p_1,p_2,p_3)$.

We consider the deformation space of complex hyperbolic triangle group  $\Delta(3,3,4)$, generated by three complex reflections $I_1, I_2$  and $I_3$. As an abstract group, it is given by 
$$\la \sigma_1,\sigma_2,\sigma_3\mid \sigma_1^2=\sigma_2^2=\sigma_3^2=(\sigma_1\sigma_2)^4=(\sigma_1\sigma_3)^3= (\sigma_2\sigma_3)^3=id\ra.$$

We will describe a parametrization  of the deformation space of $\Delta(3,3,4)$, which is a little different from that in \cite{pwx}.

Suppose that the polar vectors $n_1$, $n_2$ of the complex lines $C_1$, $C_2$ are given by 
\begin{equation*}
n_1=\left[
\begin{array}{c}
	0 \\
	1 \\
	0 \\
\end{array}
\right]\quad {\rm and}\quad  n_2=\left[
\begin{array}{c}
1/\sqrt{2} \\
1/\sqrt{2} \\
0 \\
\end{array}
\right].
\end{equation*}
Then the corresponding complex reflections $I_1$  and  $I_2$ are given by 
\begin{equation} \label{mtrixs:I1I2}
I_1=\left[
\begin{array}{ccc}
	0 & 0 & 1 \\
	0 & 1 & 0 \\
	1 & 0 & 0 \\
\end{array}
\right]\quad {\rm and}\quad I_2=\left[
\begin{array}{ccc}
	0 & 0 & 1 \\
	0 & 1 & 0 \\
	1 & 0 & 0 \\
\end{array}
\right].
\end{equation} 

We may also suppose that the polar vector $n_3$ of $C_3$ is
\begin{equation*}
n_3=\left[
\begin{array}{c}
	a \\
	b e^{i\theta} \\
	d \\
\end{array}
\right].
\end{equation*}
Furthermore, we can assume that $a$, $b$, $d$ are nonnegative real numbers by conjugating a diagonal map ${\rm Diag}(e^{i\beta}, e^{i\beta}, e^{-2i\beta})$ if necessary.
After a normalization of $n_3$, we have
\begin{equation*}
a^2+b^2-d^2=1.
\end{equation*}
The matrix for the complex reflection $I_3$ is given by 
\begin{equation}\label{mtrixs:I3}
I_3=\left[
\begin{array}{ccc}
	a^2-b^2+d^2 & 2a b e^{i\theta}& 2 a d \\
	2a b e^{-i\theta} & -a^2+b^2+d^2 & 2 b d e^{-i\theta} \\
	-2 a d & -2 b d e^{i\theta} & -a^2-b^2-d^2 \\
\end{array}
\right].
\end{equation} 
One may always assume $\theta\in [0, \pi]$ by complex conjugating if necessary.

The condition that $I_1I_3$ and $I_2I_3$ have order 3 is equivalent to $tr(I_1I_3)=tr(I_2I_3)=0$.  That is,
$$-a^2+3b^2+d^2=0$$
and
$$4 ab\cos{\theta}+a^2-b^2+d^2=0.$$
Since we know that $a^2+b^2=d^2+1$, we have
\begin{equation}\label{eq:abd}
b=1/2, \quad 2a\cos{\theta}=1/2-2a^2, \quad d^2=(4a^2-3)/4.
\end{equation} 

We also have that $d^2$ is nonnegative and $\left| 1/2-2a^2\right|\leq 2a $ if and only if $ \sqrt{3}/2\leq a\leq (\sqrt{2}+1)/2$. In other words, our parametrization of the deformation space of $\Delta(3,3,4)$  is given by $$a \in[\sqrt{3}/2, (\sqrt{2}+1)/2].$$
 In particular, the entries of $n_3$ are  all real when $a= (\sqrt{2}+1)/2$.  Thus the complex hyperbolic triangle group  $\Delta(3,3,4)$ lies in $SO(2, 1)$ when  $a= (\sqrt{2}+1)/2$.

\begin{prop}
Let $I_1, I_2$ and $I_3$ be given by \eqref{mtrixs:I1I2} and \eqref{mtrixs:I3}. Suppose  $I_1I_3$  and  $I_2I_2$ have order 3. Then  $I_1I_3I_2I_3$ is elliptic if and only if $a<1$.
\end{prop}
\begin{proof}
	We compute the trace of $I_1I_3I_2I_3$  directly and have that
\begin{eqnarray*}
	&&tr(I_1I_3I_2I_3)	\\
&=&-8ab(a^2-b^2-d^2)\cos(\theta)+a^4+a^2(2b^2-2d^2)+b^4+6b^2d^2+d^4\\
&=& 4a^2-1.
\end{eqnarray*}
The condition that $I_1I_3I_2I_3$ is elliptic is equivalent to $tr(I_1I_3I_2I_3)<3$.
\end{proof}
So our parameter space for $\left\langle I_1, I_2,I_3\right\rangle $ with $I_1I_3I_2I_3$ nonelliptic is given by
\begin{equation}\label{eq:a}
1\leq a\leq  (\sqrt{2}+1)/2.
\end{equation} 

To make the computation simpler,  we  write $a=\frac{1}{2\sqrt{1-2t}}$.  Thus the  parameter space for the triangle group  $\Delta(3,3,4)$ becomes
\begin{equation}\label{eq:a}
	3/8\leq t\leq  \sqrt{2}-1
\end{equation} with the new  parameter $t$.

Most calculations are carried out in the Siegel model. From now on, we will work on this model. 

It is convenient to introduce some notations that are used throughout the paper.  We define
\begin{equation} 
\begin{aligned} 	
		&a(t)=\sqrt{6t-2},\\
		&b(t)=\sqrt{-t^2-2t+1},\\
		&c(t)=\sqrt{\frac{8t-3}{4t-1-2t\sqrt{6t-2}}}.	
\end{aligned} 		
\end{equation} 

 In the Siegel model,  the polar vectors $n_1, n_2$ and $n_3$  are given by 
\begin{equation*} \label{mtrixs:n1n2n3}
	n_1=\left[
	\begin{array}{c}
		0 \\
	1 \\
		0 \\
	\end{array}
	\right],\,\ n_2=\left[
	\begin{array}{c}
	1/2\\
	\sqrt{2}/2 \\
1/2\\
	\end{array}
	\right],\,\ n_3=	\frac{1}{2\sqrt{1-2t}}\left[
	\begin{array}{c}
		\frac{\sqrt{2}}{2}(a(t)+1) \\
		-t+ib(t) \\
		\frac{\sqrt{2}}{2}(a(t)-1) \\
	\end{array}
	\right].
\end{equation*}

The corresponding complex reflections $I_1, I_2$ and $I_3$ are given by the matrices
\begin{equation*} \label{mtrixs:I1I2I3}
	I_1=\left[
	\begin{array}{ccc}
		-1 & 0 & 0 \\
		0 & 1 & 0 \\
		0 & 0 & -1 \\
	\end{array}
	\right],\quad I_2=\left[
	\begin{array}{ccc}
		-1/2 & \sqrt{2}/2 & 1/2 \\
		\sqrt{2}/2 & 0 & \sqrt{2}/2 \\
		1/2 & \sqrt{2}/2 & -1/2 \\
	\end{array}
	\right],
\end{equation*} 
and 
\begin{equation*} \label{mtrixs:I1I2I3}
 I_3=\left[
	\begin{array}{ccc}
		-\frac{1}{4} &\frac{ \sqrt{2}(1+a(t))(t+i b(t))}{8t-4} & \frac{1-6t-2a(t)}{8t-4} \\
	\frac{	\sqrt{2}(a(t)-1)(-t+i b(t))}{8t-4} & 	-\frac{1}{2} &	\frac{\sqrt{2}(1+a(t))(-t+i b(t))}{8t-4} \\
		\frac{1-6t+2a(t)}{8t-4} & 	-\frac{\sqrt{2}(a(t)-1)(t+i b(t))}{8t-4} & -\frac{1}{4} \\
	\end{array}
	\right],
\end{equation*} 
respectively.

\section{The Dirichlet domain of the  triangle group  $\Delta(3,3,4)$}
\subsection{The Dirichlet domain}
For the convenience of the reader we recall the construction of the Dirichlet domain of the  triangle group  $\Delta(3,3,4)$ from \cite{pwx} without proof.
The notations used here differ slightly from the notations used in \cite{pwx}.

For $k\in \Z, 1\leq k\leq 8$,  the involution $A_k$ is denoted by $$(I_2I_1)^{(k-1)/2}I_3(I_1I_2)^{(k-1)/2}$$  if $k$ is  an odd integer and $$(I_2I_1)^{(k-2)/2}I_2I_3I_2(I_1I_2)^{(k-2)/2}$$ if $k$ is   even.
One may take the index $k$ mod 8.
 Let $p_0$ be the fixed point of $I_2I_1$ in $\HdC$.  The bisector $\B_k$ is defined
to be the bisector equidistant from $p_0$ and $A_k(p_0)$. We define a polyhedron $D$ bounded by sides contained in these eight bisectors.

The combinatorial configuration of the
bisectors  as $t$ decreases from $\sqrt{2}-1$ to $3/8$  are described as follows.

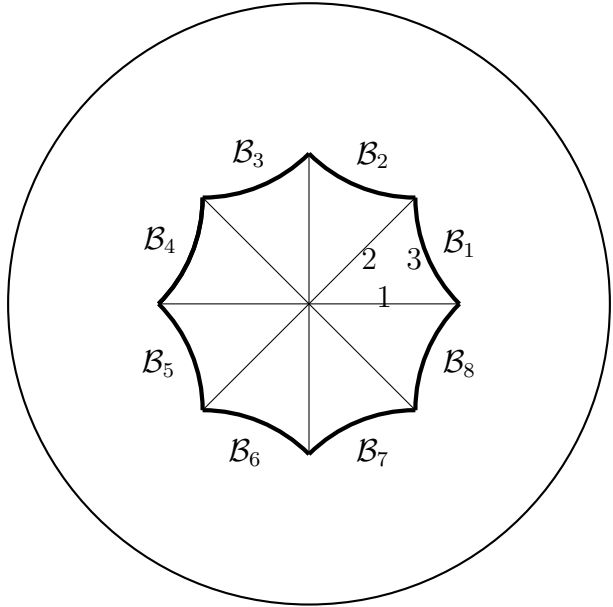
\begin{figure}
	\begin{center}
		\begin{tikzpicture}
		\draw [ thick] (0,0) circle [radius=4];;
		\draw (-2,0) -- (2,0);
		\draw (0,-2) -- (0,2);
		\draw (-1.414,-1.414) -- (1.414,1.414);
		\draw (-1.414,1.414) -- (1.414,-1.414);
		\draw [ultra thick]  (1.414,1.414) arc [radius=2, start angle=180, end angle=225];
		\draw [ultra thick]  (-1.414,1.414) arc [radius=2, start angle=270, end angle=315];
		\draw [ultra thick]  (0,2) arc [radius=2, start angle=225, end angle=270];
		\draw [ultra thick]  (-2,0) arc [radius=2, start angle=315, end angle=360];
		\draw [ultra thick]  (-2,0) arc [radius=2, start angle=315, end angle=360];
		\draw [ultra thick]  (-1.414,-1.414) arc [radius=2, start angle=0, end angle=45];
		\draw [ultra thick]  (0,-2) arc [radius=2, start angle=45, end angle=90];
		\draw [ultra thick]  (1.414,-1.414) arc [radius=2, start angle=90, end angle=135];
		\draw [ultra thick]  (2,0) arc [radius=2, start angle=135, end angle=180];
	\node at (2,0.8){$\B_1$};	
	\node at (0.8485281372, 1.979898987){$\B_2$};
	\node at (-0.8, 2){$\B_3$};
	\node at (-1.979898987, 0.8485281372){$\B_4$};
	\node at (-2,- 0.8){$\B_5$};
	\node at (-0.8485281372,-1.979898987){$\B_6$};
	\node at (0.8485281372, -1.979898987){$\B_7$};
	\node at (2,-0.8){$\B_8$};
	\node at (1,0.1){$1$};
	\node at (0.8,0.6){$2$};
	\node at (1.4,0.6){$3$};
	\end{tikzpicture}
	\end{center}
	\caption{A schematic view of Dirichlet domain of the  triangle group  $\Delta(3,3,4)$ in the ball model.}
	\label{fig:DM}
\end{figure}
 
 \begin{prop} [Parker-Wang-Xie \cite{pwx}] \label{prop:eight-bisectors}
 Let $\B_k$ be defined as above. Suppose that  $3/8\leq t\leq \sqrt{2}-1$.
 Then for each $k\in \Z/8\Z$:
 \begin{enumerate} 
 	\item The bisector $\B_k$ intersects $\B_{k\pm 1}$ in a Giraud disc. The Giraud disc is preserved by $A_{k}A_{k\pm 1}$, which has order 3.
 	\item The intersection of  $\B_k$ with $\B_{k\pm 2}$ is contained in the halfspace bounded by $\B_{k\pm 1}$ not containing $p_0$.
 	\item The bisector $\B_k$ does not intersect $\B_{k\pm \iota}$ for $3\leq \iota\leq 4$. Moreover, the boundaries of these bisectors are disjoint except for $\iota=3$ and $t=3/8$, in which case the boundaries
 	intersect in a single point, which is a parabolic fixed point.
\end{enumerate} 	
 \end{prop}

\paragraph{{\bf The symmetry for $D_t$}.}  For each $k$ mod 8 and each $n$ mod 4, we have
\begin{enumerate} 
 	\item $(I_2I_1)^{n}(\B_k)=\B_{2n+k}$;
 	\item $(I_2I_1)^{n}I_2(\B_k)=\B_{2n+3-k}$.
\end{enumerate}

 Furthermore, one can check that the side pairing maps $A_k$ for $D_t$
 satisfies the conditions of the Poincar\'e polyhedron theorem for coset 
 decomposition. Thus we have
 \begin{thm} [Parker-Wang-Xie \cite{pwx}]
 	Suppose  that $3/8\leq t\leq \sqrt{2}-1$. Let $D_t$ be the polyhedron in 
 $\HdC$ containing $p_0$ and bounded by the eight bisectors $\B_k$. Then $D_t$ is the fundamental polyhedron of triangle group  $\Delta(3,3,4)$.	
  \end{thm}

 Let $\Gamma_t$ be the even subgroup of the triangle group $\left\langle I_1, I_2,I_3\right\rangle $. Let 
 \begin{equation*} 
 	g_1=I_3I_2I_1I_2,\quad g_2=I_2I_1,\quad g_3=I_1I_2I_3I_2=g_2^{-1}g_1g_2.
 \end{equation*} 
Then  
 $$\Gamma_t=\left\langle g_1, g_2 \right\rangle. $$
Note that
$$g_3=g_2^{-1}g_1g_2, g_1=g_2g_1(g_2^{-1}g_3)(g_2g_1)^{-1}.$$
 
 For  $1\leq k\leq 8$, we have
 \begin{enumerate} 
 	\item $\B_k=\B\left(p_0, g_2^{(k-2)/2}g_3^{-1}(p_0)\right)$ if $k$ is even;
 	\item $\B_k=\B\left(p_0,  g_2^{(k-1)/2}g_1(p_0)\right)$ if $k$ is odd.
 \end{enumerate} 

\paragraph{{\bf The side-pairing maps}.} From above, it is easy to  check that $g_1$  maps the side on $\B_4$ to the side on  $\B_1$.  Side-pairing maps for other sides can be obtained from this one by symmetry.

 The Poincar\'e polyhedron theorem also shows that  $D_t$ is a fundamental 
 domain for the action of $\Gamma_t$ modulo the action of a cyclic group $\left\langle g_2\right\rangle $ of order 4.
 
 When $t=3/8$, the geometry of the group $\Gamma_{3/8}$ had been studied in \cite{der-fal}. It is the holonomy representation of a uniformizable  spherical CR structure on  the  figure-eight knot complement.

In order to study the manifold at infinity, ie the quotient of the domain of
discontinuity under the action of group.  The basis idea is to consider the 
intersection with $\partial\HdC$ of a fundamental domain for the action on 
$\HdC$.

The combinatorial structure of $\partial_{\infty} D_t=D_t\cap \partial\HdC$ is simple due to the   combinatorial structure of $D_t$. Let  $\mathcal{S}_i$ be the  spinal sphere corresponding to the bisector $\B_i$.   We define
$$ \mathcal{A}_i=\mathcal{S}_i \cap \partial_{\infty} D_t.$$
From Proposition \ref{prop:eight-bisectors}, it is easy to see that $\mathcal{A}_i$ is an annulus and  $\partial_{\infty} D_t$ is bounded by eight (pairwise isometric) annuluses.

\section{$\mbox{CR}$-spherical uniformizations for the $\R$-Fuchsian representation}  \label{section:Rfuchsian}
In this section, we just focus on the $\R$-Fuchsian representation.
Let  $t_0=\sqrt{2}-1$.  Then  $\Gamma_{t_0} \subset \mbox{\rm PO(2,1)}\subset\Pu $. 	Let $u_0=\sqrt{3\sqrt{2}-4}$ and 
$v_0=\sqrt{2\sqrt{2}-1}$.
The generators $g_1, g_2$ and $g_3$ are given  by the matrices

\begin{eqnarray*}	
	g_1&=&\left[
\begin{array}{ccc}
	\frac{3+4\sqrt{2}+6\sqrt{2}u_0+8u_0}{4} &\frac{2\sqrt{2}u_0+2u_0+2+\sqrt{2}}{4} & -\frac{1}{4} \\
	\frac{-2\sqrt{2}u_0-2u_0-2-\sqrt{2}}{4}& \frac{1}{2} &\frac{2\sqrt{2}u_0+2u_0-2-\sqrt{2}}{4} \\
	-\frac{1}{4}&\frac{2\sqrt{2}u_0+2u_0+2+\sqrt{2}}{4} & 		\frac{3+4\sqrt{2}+6\sqrt{2}u_0+8u_0}{4} \\
\end{array}
\right],\\	
g_2&=&\left[
\begin{array}{ccc}
	1/2 & \sqrt{2}/2 & -1/2 \\
	-\sqrt{2}/2 & 0 & -\sqrt{2}/2 \\
	-1/2 & \sqrt{2}/2 & 1/2 \\
\end{array}
\right],\\
g_3&=&\left[
\begin{array}{ccc}
	\frac{3+2\sqrt{2}}{4} &\frac{\sqrt{2}+2+6u_0+2\sqrt{2}u_0}{4} & \frac{-1-2\sqrt{2}-4u_0-2\sqrt{2}u_0}{4} \\
	\frac{-\sqrt{2}+2+6u_0+2\sqrt{2}u_0}{4} & 	\frac{1+2\sqrt{2}}{2} &\frac{\sqrt{2}+2+6u_0+2\sqrt{2}u_0}{4} \\
\frac{-1-2\sqrt{2}-4u_0-2\sqrt{2}u_0}{4} & 		\frac{\sqrt{2}+2+6u_0+2\sqrt{2}u_0}{4} & 	\frac{3+2\sqrt{2}}{4} \\
\end{array}
\right].
\end{eqnarray*}


Now $g_1$ is a loxodromic element in $\Pu$. Let $p_1$ and $p_2$ be the attractive and repulsive fixed points of $g_1$. We denote by $\alpha_1$  the arc of $\C$-circle
$\arc {p_1}{p_2}$. Then $\alpha_1$ is the axis at infinity of  $g_1$.  Let $\Lambda_{t_0} $ be the limit set of  $\Gamma_{t_0}$. Then it is a round circle. The crown associated to $g_1$ is the subset of $\mathbb{S}^3$ defined as 
$${\rm Crown}_{\Gamma_{t_0},g_1}= \Lambda_{t_0} \cup\Bigl(\bigcup_{g\in\Gamma_{t_0}} g\cdot\alpha_1\Bigr).$$
We denote $\Omega_{\Gamma_{t_0},g_1}\subset \Omega_{\Gamma_{t_0}}$ the complement of ${\rm Crown}_{\Gamma_{t_0},g_1}$ in $\mathbb{S}^3$. Dehornoy showed

	\begin{prop} [Dehornoy \cite{Dehornoy}] \label{thm:Ffuchsianfigure8} $\Omega_{\Gamma_{t_0},g_1}/\Gamma_{t_0}$ is homeomorphic to  the  figure-eight knot complement.
\end{prop}

 We will  reinterpret Proposition \ref{thm:Ffuchsianfigure8} by using the fundamental domain.

Note that $g_2^{-1}g_3$ is also a loxodromic element in $\Pu$. Let $q_1$ and $q_2$ be the attractive and repulsive fixed points of  $g_2^{-1}g_3$. We denote by $\beta_1$  the arc of $\C$-circle $\arc {q_1}{q_2}$. Then $\beta_1$ is the axis at infinity of  $g_2^{-1}g_3$.  Define
  $$\alpha_i=g_2^{i}(\alpha_1),\quad \beta_i=g_2^{i}(\beta_1),$$
for $i=2,3,4.$

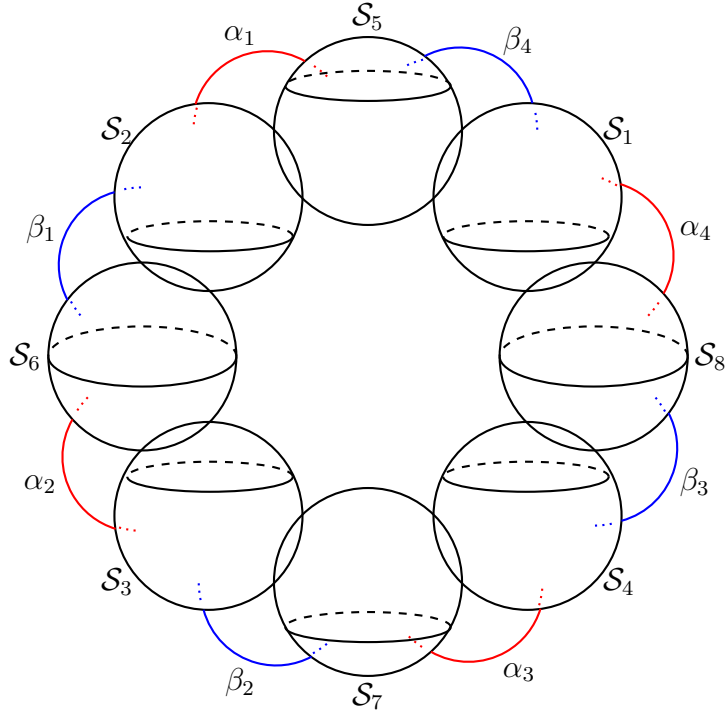
\begin{figure}
	\begin{center}
		\begin{tikzpicture}
		\draw[thick] (3,0) circle (1.25);
		\draw[thick] (-3,0) circle (1.25);
		\draw[thick] (0,3) circle (1.25);
		\draw[thick] (0,-3) circle (1.25);
		\draw[thick] (2.121,2.121) circle (1.25);
		\draw[thick] (-2.121,2.121) circle (1.25);
		\draw[thick] (-2.121,-2.121) circle (1.25);
		\draw[thick] (2.121,-2.121) circle (1.25);
		
		\draw[thick] (-4.25,0) arc [start angle=180, end angle=360,x radius=1.25, y radius=0.4];
		\draw[thick,dashed] (-1.75,0) arc [start angle=0, end angle=180,x radius=1.25, y radius=0.4];
		\draw[thick] (1.75,0) arc [start angle=180, end angle=360,x radius=1.25, y radius=0.4];
		\draw[thick,dashed] (4.25,0) arc [start angle=0, end angle=180,x radius=1.25, y radius=0.4];

			\draw[thick] (-1.1,3.6) arc [start angle=180, end angle=360,x radius=1.1, y radius=0.2];
		\draw[thick,dashed] (1.1,3.6) arc [start angle=0, end angle=180,x radius=1.1, y radius=0.2];
		
			\draw[thick] (-1.1,-3.6) arc [start angle=180, end angle=360,x radius=1.1, y radius=0.2];
		\draw[thick,dashed] (1.1,-3.6) arc [start angle=0, end angle=180,x radius=1.1, y radius=0.2];
		
			\draw[thick] (-3.2,1.6) arc [start angle=180, end angle=360,x radius=1.1, y radius=0.2];
		\draw[thick,dashed] (-1,1.6) arc [start angle=0, end angle=180,x radius=1.1, y radius=0.2];
		
			\draw[thick] (-3.2,-1.6) arc [start angle=180, end angle=360,x radius=1.1, y radius=0.2];
		\draw[thick,dashed] (-1,-1.6) arc [start angle=0, end angle=180,x radius=1.1, y radius=0.2];
		
		\draw[thick] (1,1.6) arc [start angle=180, end angle=360,x radius=1.1, y radius=0.2];
		\draw[thick,dashed] (3.2,1.6) arc [start angle=0, end angle=180,x radius=1.1, y radius=0.2];
		\draw[thick] (1,-1.6) arc [start angle=180, end angle=360,x radius=1.1, y radius=0.2];
		\draw[thick,dashed] (3.2,-1.6) arc [start angle=0, end angle=180,x radius=1.1, y radius=0.2];

		\draw[ thick, red] (-0.8371127315,3.928300746) arc [start angle=60, end angle=163, radius=1];
	  \draw[thick, red, dotted] (-0.5427712259,3.724570095) arc [start angle=45, end angle=60, radius=1];
	  \draw[thick, red,dotted] (-2.275,3.309835018) arc [start angle=163, end angle=178, radius=1];

	\draw[ thick, blue] (-3.369656184,2.185800006) arc [start angle=105, end angle=208, radius=1];
	\draw[thick, blue, dotted] (-3.017465985,2.249871556) arc [start angle=90, end angle=105, radius=1];
	\draw[thick, blue,dotted] (-3.949074712,0.7317388584) arc [start angle=208, end angle=223, radius=1];
	
		\draw[ thick, red] (-3.928300743,-0.8371127312) arc [start angle=150, end angle=253, radius=1];
		\draw[thick, red, dotted] (-3.724570093, -0.5427712258) arc [start angle=135, end angle=150, radius=1];
		\draw[thick, red,dotted] (-3.309835016, -2.274999999) arc [start angle=253, end angle=268, radius=1];
		
		\draw[ thick, blue] (-2.185800004, -3.369656182) arc [start angle=195, end angle=298, radius=1];
	\draw[thick, blue, dotted] (-2.249871555, -3.017465983) arc [start angle=180, end angle=195, radius=1];
	\draw[thick, blue,dotted] (-0.7317388577, -3.949074710) arc [start angle=298, end angle=313, radius=1];
		
		\draw[ thick, red] (0.8371127312, -3.928300741) arc [start angle=240, end angle=343, radius=1];
			\draw[thick, red, dotted] (0.5427712251, -3.724570091) arc [start angle=225, end angle=240, radius=1];
		\draw[thick, red,dotted] (2.274999998, -3.309835014) arc [start angle=343, end angle=358, radius=1];
		
		\draw[ thick, blue ] (3.369656180, -2.185800003) arc [start angle=285, end angle=388, radius=1];
		\draw[thick,  blue, dotted] (3.017465981, -2.249871554) arc [start angle=270, end angle=285, radius=1];
		\draw[thick,  blue,dotted] (3.949074708, -0.731738857) arc [start angle=388, end angle=403, radius=1];
		
		\draw[ thick, red ] (3.928300739,0.8371127305) arc [start angle=330, end angle=433, radius=1];
		\draw[thick,  red, dotted] (3.724570089,0.5427712244) arc [start angle=315, end angle=330, radius=1];
		\draw[thick,  red,dotted] (3.309835012,2.274999997) arc [start angle=418, end angle=433, radius=1];
		
	\draw[ thick, blue ] (2.185800002,3.369656179) arc [start angle=375, end angle=478, radius=1];
		\draw[thick,  blue, dotted] (2.249871553,3.017465980) arc [start angle=360, end angle=375, radius=1];
	\draw[thick,  blue,dotted] (0.7317388563,3.949074706) arc [start angle=463, end angle=478, radius=1];	
		
\node[above] at (0,4.2) {$\mathcal{S}_{5}$};
\node[left] at (-2.98,3) {$\mathcal{S}_{2}$};	
\node[left] at (-4.2,0) {$\mathcal{S}_{6}$};	
\node[left] at (-2.98,-3) {$\mathcal{S}_{3}$};
\node[below] at (0,-4.2) {$\mathcal{S}_{7}$};		
\node[right] at (2.98,-3) {$\mathcal{S}_{4}$};
\node[right] at (4.2,0) {$\mathcal{S}_{8}$};
\node[right] at (2.98,3) {$\mathcal{S}_{1}$};
\node[above] at (-1.7,4) {$\alpha_1$};
\node[left] at (-4,-1.7){$\alpha_2$};	
\node[right] at (1.65,-4.2){$\alpha_3$};	
\node[right] at (4,1.67){$\alpha_4$};	
\node[left] at (-4,1.7){$\beta_1$};	
\node[below] at (-1.7,-4) {$\beta_2$};	
\node[right] at (4,-1.67){$\beta_3$};	
\node[right] at (1.65,4.2){$\beta_4$};	
		\end{tikzpicture}
	\end{center}
	\caption{A schematic view of the configuration of the eight spinal spheres and the eight $\C$-arcs. Each round sphere is a spinal sphere, and $\partial_{\infty} D_{t}$ is the region outside all the spinal spheres.  $\alpha_i$($\beta_i$) is the thick red (blue) arc with end points in some of the spinal spheres.}
	\label{fig:configuration1}
\end{figure}

\begin{figure}[htbp]
	\centering
	\includegraphics[width=6cm]{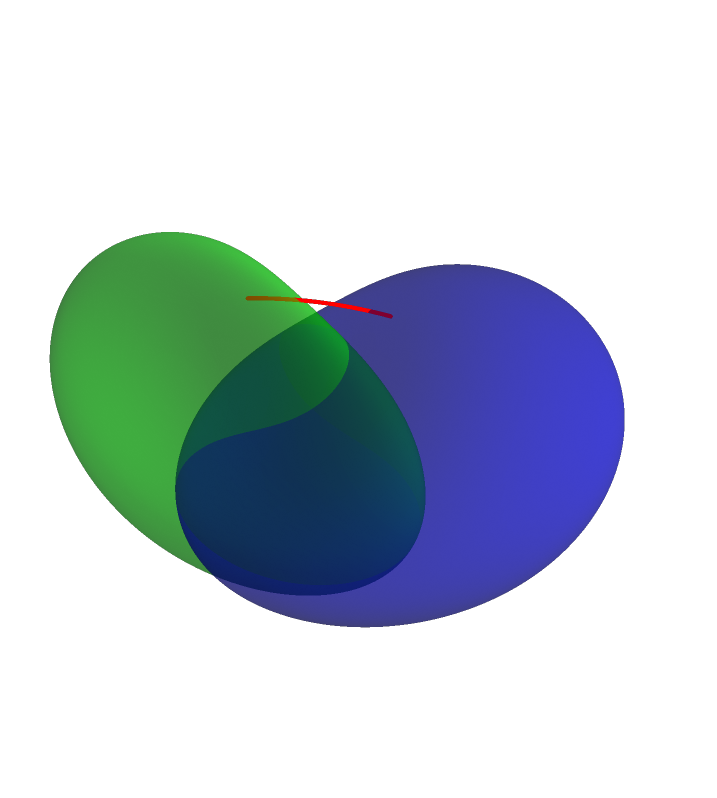}
	\caption{The $\C$-arc $\hat{\alpha_1}$ is the part of $\C$-arc $\alpha_1$, which is the axis  at infinity of $g_1$.  It lies in $\partial_{\infty} D_{t}$ with end points on the spinal spheres  $\mathcal{S}_5$(the green one) and $\mathcal{S}_2$(the blue one). }
	\label{fig:arc-alpha1}
\end{figure}

\begin{figure}[htbp]
	\centering
	\includegraphics[width=10cm]{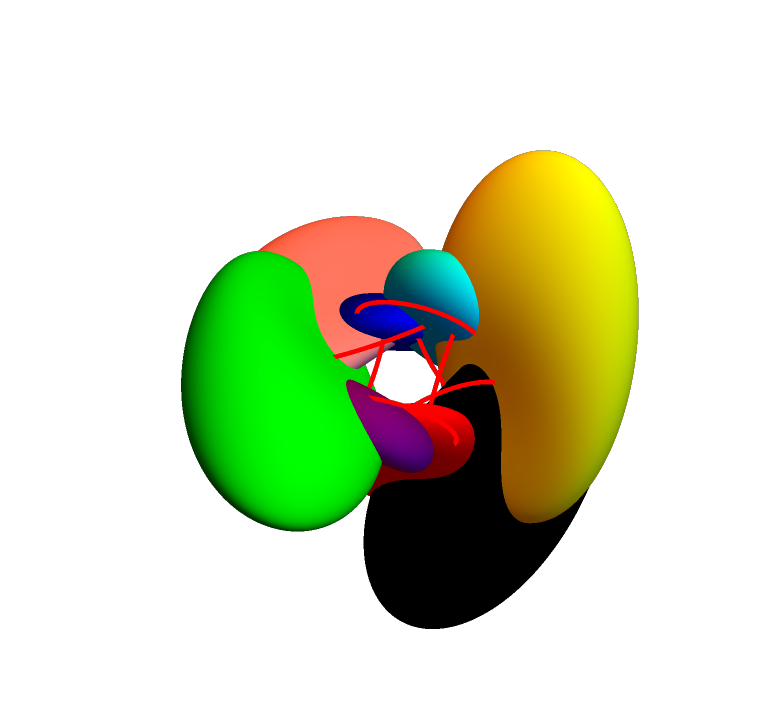}
	\caption{A realistic view of the configuration of the eight spinal spheres and the eight $\C$-arcs. }
	\label{fig:configuration2}
\end{figure}

See Figure 	\ref{fig:configuration1} for a  schematic view of the configuration of the eight spinal spheres and the eight $\C$-arcs. Figure 	\ref{fig:configuration1} should also be compared with Figure 	\ref{fig:configuration2}. 

\subsection{The configuration of the eight $\C$-arcs}
We will  study the intersections of arcs $\alpha_i, \beta_i$ with the spinal spheres $\mathcal{S}_i$.  Let $U_{t_0}=\partial_{\infty }D_{t_0} $. Then $U_{t_0}$ is a solid torus. We denote by $\hat{\alpha_i}, \hat{\beta_i}$ the segments of the arcs $\alpha_i, \beta_i$, which are contained in the interior of solid torus  $U_{t_0}$.  We define $\hat{\alpha_1}^{-}$($\hat{\alpha_1}^{+}$) to be the end point of $\hat{\alpha_1}$ close to the attractive fixed  point $p_1$(repulsive fixed  point $p_2$) of  $\alpha_1$, see Figure \ref{fig:arc-alpha1}. The end point of $\hat{\beta_1}$ can be defined
similarly. Note that 
$$\hat{\alpha_i}=g_2^{i}(\hat{\alpha_1}),\quad \hat{\beta_i}=g_2^{i}(\hat{\beta_1}),$$
for $i=2,3,4.$

\begin{prop}\label{prop:endpointalpha1}
	 The end point $\hat{\alpha_4}^{-}$ of  $\hat{\alpha_4}$ is on the spinal sphere $\mathcal{A}_1$, and the other end point $\hat{\alpha_4}^{+}$ of  $\hat{\alpha_4}$ is on the spinal sphere  $\mathcal{A}_8$.
\end{prop}

\begin{proof}	
 Firstly, we claim that the $\C$-arc $\alpha_4$ is contained in the exterior of the spinal spheres $\mathcal{S}_i(1\leq i\leq 8)$ except for  $\mathcal{S}_1$, $\mathcal{S}_4$, $\mathcal{S}_5$ and  $\mathcal{S}_8$.

	Note that $\alpha_4$ is the $\C$-arc, which is the axis at infinity of $g_2^{-1}g_1g_2$. Let $v$ be the polar	vector of the $\C$-circle containing the arc $\alpha_4$. Then 
	\begin{equation*}
		v=g_2^{-1}\left(n_3\boxtimes I_2(n_1)\right)=\left[
		\begin{array}{c}
			\frac{9+4\sqrt{2}+(10\sqrt{2}+12)u_0}{7}\\
			0 \\
			1 \\
		\end{array}
		\right].
	\end{equation*}

The lift of the  $\C$-arc $\alpha_4$ can be written as 
	\begin{equation}\label{eq:alpha4}
 V_1=\left[
	\begin{array}{c}
		\frac{-9-4\sqrt{2}-(10\sqrt{2}+12)u_0}{7}\\
			\frac{\sqrt{(140\sqrt{2}+168)u_0+56\sqrt{2}+126}}{7}(x+yi) \\
		1 \\
	\end{array}
	\right],
\end{equation}
where $y=\sqrt{1-x^2},-1\leq x\leq 1$. 

We claim that  the intersection of $\alpha_4$ with  $\mathcal{S}_2$ is empty.
Substituting (\ref{eq:alpha4}) to the equation of the bisector $\mathcal{B}_2$
$$|\langle  V_1, q_0 \rangle|=|\langle  V_1, g_2g_1(q_0)\rangle|.$$
We get the equation
\begin{equation*}
	\begin{aligned} 
&(260\sqrt{7}u_0x+144\sqrt{14}u_0x+76\sqrt{7}x+54\sqrt{2})\sqrt{9+4\sqrt{2}+12u_0+10\sqrt{2}u_0}\\
	&-1154\sqrt{2}-2888u_0-2024\sqrt{2}u_0-1592=0.\\
	\end{aligned}
\end{equation*}

Solve the  equation, we have
$$x=\sqrt{\frac{4816\sqrt{2}-6755}{7}}\approx-1.0676.$$
Therefore,  $\alpha_4$ does not intersect with $\mathcal{S}_2$, which is the boundary at infinity of $\mathcal{B}_2$. 
With the same argument, one can also prove that $\alpha_4$ does not intersect with $\mathcal{S}_3$, $\mathcal{S}_6$ and $\mathcal{S}_7$.

Next, we study the intersection of  $\alpha_4$ with $\mathcal{S}_1$, $\mathcal{S}_4$, $\mathcal{S}_5$  and $\mathcal{S}_8$. The intersection point divide the arc $\alpha_4$
into several segments. We will determine which segment is $\hat{\alpha_4}$.

Substituting (\ref{eq:alpha4}) to the equation of the bisector $\mathcal{B}_1$
$$|\langle  V_1, q_0 \rangle|^2=|\langle  V_1, g_1(q_0)\rangle|^2.$$
We get
\begin{equation*}
\begin{aligned} 
	&(44\sqrt{7}u_0x+32\sqrt{14}u_0x+12\sqrt{7}x+10\sqrt{14}x)\sqrt{9+4\sqrt{2}+12u_0+10\sqrt{2}u_0}\\
	&-126\sqrt{2}-280u_0-168\sqrt{2}u_0-112=0.\\
\end{aligned}
\end{equation*}

The intersection point corresponds to the solution

$$
x=\sqrt{8\sqrt{2}-11},\quad y=2\sqrt{2}-2.
$$

Substituting (\ref{eq:alpha4}) to the equation of the bisector $\mathcal{B}_4$
$$|\langle  V_1, q_0 \rangle|^2=|\langle  V_1,  g_2^3g_1(q_0)\rangle|^2.$$

We get
\begin{equation*}
\begin{aligned} 
	&(172\sqrt{7}u_0x+120\sqrt{14}u_0x-52\sqrt{7}x-34\sqrt{14}x)\sqrt{9+4\sqrt{2}+12u_0+10\sqrt{2}u_0}\\
	&-1512u_0-1064\sqrt{2}u_0-840=0.\\
\end{aligned}
\end{equation*}

The intersection point corresponds to the solution

$$
x=-\frac{\sqrt{16\sqrt{2}+13}}{7},\quad y=\frac{4\sqrt{2}-2}{7}.
$$

Substituting (\ref{eq:alpha4}) to the equation of the bisector $\mathcal{B}_5$
$$|\langle  V_1, q_0 \rangle|^2=|\langle  V_1, g_3^{-1}(q_0)\rangle|^2.$$
We get
\begin{equation*}
\begin{aligned} 
	&(172\sqrt{7}u_0x+120\sqrt{14}u_0x+52\sqrt{7}x+34\sqrt{14}x)\sqrt{9+4\sqrt{2}+12u_0+10\sqrt{2}u_0}\\
	&-602\sqrt{2}-1512u_0-1064\sqrt{2}u_0-840=0.\\
\end{aligned}
\end{equation*}

The intersection point corresponds to the solution

$$
x=\frac{\sqrt{16\sqrt{2}+13}}{7},\quad y=\frac{4\sqrt{2}-2}{7}.
$$

Substituting (\ref{eq:alpha4}) to the equation of the bisector $\mathcal{B}_8$
$$|\langle  V_1, q_0 \rangle|^2=|\langle  V_1, g_2^{-1}g_3^{-1}(q_0)\rangle|^2.$$
We get
\begin{equation*}
	\begin{aligned} 
		&(44\sqrt{7}u_0x+32\sqrt{14}u_0x-12\sqrt{7}x-10\sqrt{14}x)\sqrt{9+4\sqrt{2}+12u_0+10\sqrt{2}u_0}\\
		&-280u_0-168\sqrt{2}u_0-112=0.\\
	\end{aligned}
\end{equation*}

The intersection point corresponds to the solution

$$
x=\sqrt{8\sqrt{2}-11},\quad y=2\sqrt{2}-2.
$$

By simple calculation,
it find that one end point  of  $\alpha_4$ lies inside  $\mathcal{S}_1$ and $\mathcal{S}_5$ and the other end point  of  $\alpha_4$ lies inside  $\mathcal{S}_4$ and $\mathcal{S}_8$. 
It is also easy to check that   the intersection point of  $\alpha_4$  with $\mathcal{S}_5$ lies in  $\mathcal{S}_1$ and  the intersection point of  $\alpha_4$  with $\mathcal{S}_1$ does not lie in any spinal sphere.
So this intersection point is on $\mathcal{A}_1$.

 we also see that the intersection point of  $\alpha_4$  with $\mathcal{S}_4$ lies in   $\mathcal{S}_8$ and  the intersection point of  $\alpha_4$  with $\mathcal{S}_8$ does not lie in any spinal sphere.
 So the intersection point is on $\mathcal{A}_8$.

From the configuration of spinal spheres, we can see that the segment  on $\alpha_4$ between the intersect points of  $\alpha_4$  with $\mathcal{A}_1$ and $\mathcal{A}_8$ is the $\C$-arc $\hat{\alpha_4}$ that we are looking for,
see Figure \ref{fig:alpha4-arc}.
\end{proof}

Similarly, we have 
\begin{prop} \label{prop:endpointbeta1}
	The end point $\hat{\beta_1}^{-}$ of  $\hat{\beta_1}$ is on the spinal sphere  $\mathcal{A}_2$, and the other end point $\hat{\beta_1}^{+}$ of  $\hat{\beta_1}$ is on the spinal sphere  $\mathcal{A}_6$.
\end{prop}

From the calculations in Proposition \ref{prop:endpointalpha1} and  Proposition \ref{prop:endpointbeta1}, we have

	\begin{equation*}
\hat{\alpha_4}^{-}=\left[
	\begin{array}{c}
\frac{-9-4\sqrt{2}-12u_0-10\sqrt{2}u_0}{7}\\
	2-\sqrt{2}+2u_0+\frac{4v_0-6\sqrt{2}v_0-4u_0v_0-8\sqrt{2}u_0v_0}{7}i \\
		1 \\
	\end{array}
	\right],
\end{equation*}

\begin{equation*}
	\hat{\alpha_4}^{+}=\left[
	\begin{array}{c}
	\frac{-9-4\sqrt{2}-12u_0-10\sqrt{2}u_0}{7}\\
	-2+\sqrt{2}-2u_0+\frac{4v_0-6\sqrt{2}v_0-4u_0v_0-8\sqrt{2}u_0v_0}{7}i \\
		1 \\
	\end{array}
	\right],
\end{equation*}

\begin{equation*}
	\hat{\beta_1}^{-}=\left[
	\begin{array}{c}
	18\sqrt{2}u_0+26u_0-9\sqrt{2}-13\\
	\frac{(32530+24631\sqrt{2})u_0-12172-15731\sqrt{2}i\left((9826+6664\sqrt{2})u_0+2884\sqrt{2}+5208\right)\sqrt{(1-2u_0)(2+3\sqrt{2})}}{6689}\\
		1 \\
	\end{array}
	\right],
\end{equation*}

\begin{equation*}
	\hat{\beta_1}^{-}=\left[
\begin{array}{c}
	18\sqrt{2}u_0+26u_0-9\sqrt{2}-13\\
	\frac{(21246+16105\sqrt{2})u_0-9540-11181\sqrt{2}i\left((3006+2372\sqrt{2})u_0+1076\sqrt{2}+1128\right)\sqrt{(1-2u_0)(2+3\sqrt{2})}}{4657}\\
	1 \\
\end{array}
\right].
\end{equation*}

Under the action of $g_2$, we can obtain the end points of all $\C$-arcs  $\hat{\alpha_i}^{-}$ and  $\hat{\beta_i}^{-}$. We summarize these  in Table  \ref{tab:isometric spheres}. 

\begin{table}[!htbp]
		\caption{The positions and the coordinates of the end points of the eight arcs.}
	\centering
	\begin{tabular}{c|c}
		\toprule
		\textbf{$\C$-arc} & \textbf{End points  }\\
		\midrule
		$\hat{\alpha_1}$&$\hat{\alpha_1}^{-}\in \mathcal{A}_2$, $\hat{\alpha_1}^{+}\in \mathcal{A}_5$\\
		\hline
			$\hat{\alpha_2}$&$\hat{\alpha_2}^{-}\in \mathcal{A}_3$,  $\hat{\alpha_2}^{+}\in \mathcal{A}_6$\\
		\hline
			$\hat{\alpha_3}$&$\hat{\alpha_3}^{-}\in \mathcal{A}_4$,  $\hat{\alpha_3}^{+}\in \mathcal{A}_7$\\
		\hline
			$\hat{\alpha_4}$&$\hat{\alpha_4}^{-}\in \mathcal{A}_1$, $\hat{\alpha_4}^{+}\in \mathcal{A}_8$\\
		\hline
			$\hat{\beta_1}$&$\hat{\beta_1}^{-}\in \mathcal{A}_2$, $\hat{\beta_1}^{+}\in \mathcal{A}_6$\\
		\hline
			$\hat{\beta_2}$&$\hat{\beta_2}^{-}\in \mathcal{A}_3$, $\hat{\beta_2}^{+}\in \mathcal{A}_7$\\
		\hline
			$\hat{\beta_3}$&$\hat{\beta_3}^{-}\in \mathcal{A}_4$,  $\hat{\beta_3}^{+}\in \mathcal{A}_8$\\
		\hline
			$\hat{\beta_4}$&$\hat{\beta_4}^{-}\in \mathcal{A}_1$, $\hat{\beta_4}^{+}\in \mathcal{A}_5$\\
		\bottomrule
	\end{tabular}
	
	\label{tab:isometric spheres}
\end{table}

\subsection{The configuration of the eight cutting disks}
Recall  the affine disk bounded by a $\C$-circle in Definition \ref{definition:cdisk}. 

\begin{defn}From Propositions \ref{prop:endpointalpha1} and \ref{prop:endpointbeta1}, for each $\C$-arc $\hat{\alpha_i}$ (or $\hat{\beta_i}$), there is a affine disk bounded by the $\C$-circle containing this $\C$-arc. We define the  {\it cutting disk} to be the part of the affine disk bounded by the  $\C$-arc and two spinal spheres containing the end points of the $\C$-arc.
\end{defn}
See Figure 	\ref{fig:cut-disk} for a realistic view of the  cutting disk corresponding to $\hat{\beta_1}$.

\begin{figure}[htbp]
	\centering
	\includegraphics[width=5cm]{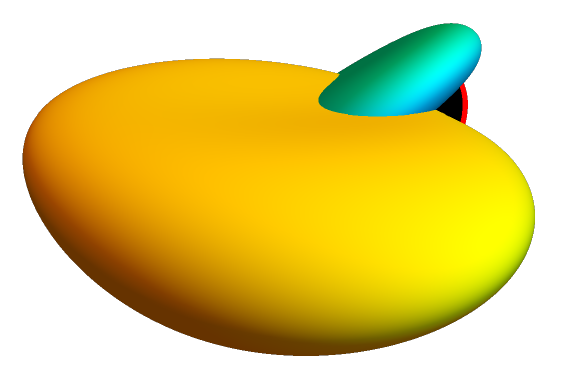}
	\caption{ The embedded cutting disk (the black one) corresponding to $\hat{\beta_1}$ is the region of the affine disk that lies outside the spinal spheres $\mathcal{S}_2$ (the yellow one), $\mathcal{S}_6$ (the cyan one) and is constrained by the $\C$-arc  $\hat{\beta_1}$.}
	\label{fig:cut-disk}
\end{figure}

From the definition, we know that each  cutting disk  properly embeds in the fundamental domain at infinity.  There are eight embedded cutting disks corresponding to the eight $\C$-arcs.   We have
 \begin{prop} The eight  cutting disks are disjoint.
\end{prop}
The proof of this  proposition for the deformation from the $\R$-Fuchsian case to the degenerate case will be given in Section 6 (see Proposition  \ref{prop:disjoint}). We include this proposition in the $\R$-Fuchsian case here just for completeness.

From some routine calculations and the  relation 
$$g_2g_1(g_2^{-1}g_3)(g_2g_1)^{-1}=g_1,$$
we can check that $g_2g_1(\hat{\beta_1})\cup \hat{\alpha_1}$ is the  fundamental domain of the  $g_1$ acting on its axis at infinity ${\alpha_1}$.
This allow us to get the following proposition.

 \begin{prop}\label{prop:CrownRfuchsian} 	$${\rm Crown}_{\Gamma_{t_0},g_1}= \Lambda_{t_0} \cup\Bigl(\bigcup_{g\in\Gamma_{t_0}} g\cdot\alpha_1\Bigr)=\Lambda_{t_0} \cup\Bigl(\bigcup_{g\in\Gamma_{t_0}} g\cdot(\hat{\alpha_1}\cup \hat{\beta_1})\Bigr).$$ 
 \end{prop}

The principal significance of  Proposition \ref{prop:CrownRfuchsian}  is that it allows us to get the figure-eight knot complement from  Dehornoy's result by applying techniques using the fundamental domain. Since $\partial_{\infty}D_{t_0}-(\cup^4_{i=1}(\hat{\alpha_i}\cup \hat{\beta_i}))$ is a subset of $\Omega_{\Gamma_{t_0},g_1}$, the restriction of the quotient map on $\Omega_{\Gamma_{t_0},g_1} \rightarrow \Omega_{\Gamma_{t_0},g_1}/\Gamma_{t_0}$ to $\partial_{\infty}D_{t_0}-(\cup^4_{i=1}(\hat{\alpha_i}\cup \hat{\beta_i}))$ also gives  a quotient space, which can be viewed as the quotient of $\partial_{\infty}D_{t_0}-(\cup^4_{i=1}(\hat{\alpha_i}\cup \hat{\beta_i}))$ by side-pairings on $$ \cup^8_{j=1} \mathcal{A}_i\cup(\cup^4_{i=1}(\hat{\alpha_i}\cup \hat{\beta_i})).$$

\begin{prop}\label{prop:sidepairRfuchsian} The quotient space of $$\partial_{\infty}D_{t_0}-(\cup^4_{i=1}(\hat{\alpha_i}\cup \hat{\beta_i}))$$ by side-pairings and the quotient space  $\Omega_{\Gamma_{t_0},g_1}/\Gamma_{t_0}$ are homeomorphic. So it is the  figure-eight knot complement.
\end{prop}
\begin{proof} We denote by $M$ the quotient space of  $\partial_{\infty}D_{t_0}-\cup^4_{i=1}(\hat{\alpha_i}\cup \hat{\beta_i})$ by side-pairings.  Then it is trivial that $M$ is a subspace of  $\Omega_{\Gamma_{t_0},g_1}/\Gamma_{t_0}$. Conversely, from the side-pairings on $ \cup^8_{j=1} \mathcal{A}_i\cup(\cup^4_{i=1}(\hat{\alpha_i}\cup \hat{\beta_i}))$, we have $M$ contains $\Omega_{\Gamma_{t_0},g_1}/\Gamma_{t_0}$. So they are homeomorphic.
		Then by Proposition  \ref{thm:Ffuchsianfigure8}, $M$  is the  figure-eight knot complement.
	
\end{proof}

\section{Geometric stability in the deformation }
In this section we focus on the group $\Gamma_{t}$ for $t\in (3/8,\sqrt{2}-1]$. The combinatorics of the fundamental domain of $\Gamma_{t}$
does not change for $t\in (3/8,\sqrt{2}-1]$ due to the work of \cite{pwx}.
Therefore, we only need to show that the configurations of the eight arcs and the eight cutting disks are the same as the $\R$-Fuchsian representation.

\begin{prop}\label{prop:arcends} The spinal spheres where the end points of eight arcs are located do not change during the deformation.
\end{prop}
\begin{proof}
	We will show that one end point of $\C$-arc $\hat{\alpha_{4,t}}$ 
is always located	on the spinal sphere $\mathcal{S}_{1,t}$ during the deformation. The proofs in other cases are similar, and we omit them.

Our proof is in three steps:
\begin{description}
	\item[Step 1] To show that the end point $p_{4,2}$  of the $\C$-arc $\alpha_{4,t}$  lies in the spinal spheres $\mathcal{S}_{1,t}$ and $\mathcal{S}_{5,t}$  and the end point $p_{4,1}$  of the $\C$-arc $\alpha_{4,t}$  lies in the spinal spheres $\mathcal{S}_{4,t}$ and $\mathcal{S}_{8,t}$.
	\item[Step 2] To show that the  $\C$-arc $\alpha_{4,t}$($\arc {p_{4,1}}{p_{4,2}}$) intersects with the
	 spinal spheres $\mathcal{S}_{1,t}$, $\mathcal{S}_{4,t}$, $\mathcal{S}_{5,t}$ and $\mathcal{S}_{8,t}$ only once and does not intersect other spinal spheres.
	\item[Step 3] Note that the spinal sphere $\mathcal{S}_{1,t}$ only intersects with two  spinal spheres, $\mathcal{S}_{5,t}$ and $\mathcal{S}_{8,t}$. In the beginning of the deformation, one end point of  $\hat{\alpha_{4,t}}$ is on the spinal sphere $\mathcal{A}_{1,t}$, see Figure \ref{fig:alpha4-arc}. 
	If the configuration in  Figure \ref{fig:alpha4-arc} turns into the configurations in Figure \ref{fig:alpha4-arc2}, then the $\C$-arc $\alpha_{4,t}$ will pass through $\mathcal{A}_{1,t}\cap \mathcal{A}_{5,t}$ or $\mathcal{A}_{1,t}\cap \mathcal{A}_{8,t}$ at some time  during the deformation by a geometric continuity argument. We will show that this is impossible.
\end{description}

We begin with the Step 1.  By a simple calculation, we have

\begin{equation}\label{eq:coor-of-p1p2}
	p_{4,1}=\left[
	\begin{array}{c}
	-\frac{\sqrt{2}(-t-\sqrt{6t-2}+\sqrt{1-2t-t^2}i)}{2\sqrt{6t-2}} \\
		\sqrt{\frac{8t-3}{3t-1}} \\
			-\frac{\sqrt{2}(-t+\sqrt{6t-2}+\sqrt{1-2t-t^2}i)}{2\sqrt{6t-2}}\\
	\end{array}
	\right],\quad p_{4,2}=\left[
	\begin{array}{c}
		-\frac{\sqrt{2}(-t-\sqrt{6t-2}+\sqrt{1-2t-t^2}i)}{2\sqrt{6t-2}} \\
		-\sqrt{\frac{8t-3}{3t-1}} \\
		-\frac{\sqrt{2}(-t+\sqrt{6t-2}+\sqrt{1-2t-t^2}i)}{2\sqrt{6t-2}}\\
	\end{array}
	\right].
\end{equation}

Substitute  (\ref{eq:coor-of-p1p2}) to the equations of the bisectors of $\B_{1,t}$ and $\B_{5,t}$. Then we get
\begin{equation*}
	\begin{aligned}	
		&|\langle  p_{4,2}, q_0 \rangle|^2-|\langle p_{4,2}, g_2g_1(q_0)\rangle|^2=\frac{8t-3-\sqrt{8t-3}}{4t-2},\\
		&|\langle p_{4,2}, q_0 \rangle|^2-|\langle  p_{4,2}, g_3^{-1}(q_0)\rangle|^2=\frac{(4t-1)\sqrt{8t-3}-8t+3}{2(2t-1)^2}.
	\end{aligned}
\end{equation*}

It is easy to check that 
$$\frac{8t-3-\sqrt{8t-3}}{4t-2}>0,\quad \frac{(4t-1)\sqrt{8t-3}-8t+3}{2(2t-1)^2}>0$$
	for $t\in (3/8,\sqrt{2}-1]$.
	
	That is, the point  $p_{4,2}$ lies inside the spinal sphere $\mathcal{S}_{1,t}$ and $\mathcal{S}_{5,t}$.
	
	Substitute  (\ref{eq:coor-of-p1p2}) to the equations of the bisectors of $\B_{4,t}$ and $\B_{8,t}$. Then we get
	
	\begin{equation*}
		\begin{aligned}	
			&|\langle  p_{4,2}, q_0 \rangle|^2-|\langle p_{4,2}, g_2^{-1}g_3^{-1}(q_0)\rangle|^2=\frac{8t-3+\sqrt{8t-3}}{4t-2},\\
			&|\langle p_{4,2}, q_0 \rangle|^2-|\langle  p_{4,2}, g_2^{-1}g_1(q_0)\rangle|^2=\frac{(1-4t)\sqrt{8t-3}-8t+3}{2(2t-1)^2}.
		\end{aligned}
	\end{equation*}
Both equations' right sides are  negative for $t\in (3/8,\sqrt{2}-1]$. So the point  $p_{4,2}$ lie outside the spinal sphere $\mathcal{S}_{4,t}$ and $\mathcal{S}_{8,t}$.

Similarly, we can show that  the point  $p_{4,1}$ lies inside the spinal sphere $\mathcal{S}_{4,t}$ and $\mathcal{S}_{8,t}$ and lies outside the spinal sphere $\mathcal{S}_{1,t}$ and $\mathcal{S}_{5,t}$.

Next, we will complete the Step 2.

The polar vector of $\alpha_{4,t}$ is given by 
\begin{equation*}
\left[
\begin{array}{c}
	\frac{\sqrt{6t-2}+t+i\sqrt{1-2t-t^2}}{\sqrt{6t-2}-t-i\sqrt{1-2t-t^2}} \\
 0\\
	1\\
\end{array}
\right].
\end{equation*}

To make the calculation simpler, we apply the following transformation
$$	T_1=\left[
\begin{array}{ccc}
	\sqrt{4t-1-2t\sqrt{6t-2}} & 0 & -\frac{2\sqrt{6t-2}\sqrt{-t^2-2t+1}i}{\sqrt{4t-1-2t\sqrt{6t-2}}} \\
	0 & 1 & 0 \\
	0 & 0 & \frac{1}{\sqrt{4t-1-2t\sqrt{6t-2}}} \\
\end{array}
\right].
$$

Then the polar vector of $\alpha_{4,t}$ is given by 
\begin{equation*}
	\left[
	\begin{array}{c}
		8t-3 \\
		0\\
		1\\
	\end{array}
	\right].
\end{equation*}

The lift of the  $\C$-arc  $T_1(\alpha_{4,t})$ can be written as 
\begin{equation}\label{eq:Talpha4}
	V_t=\left[
	\begin{array}{c}
			3-8t\\
		\sqrt{2}\sqrt{8t-3} (x+yi) \\
		1 \\
	\end{array}
	\right],
\end{equation}
where $y=-\sqrt{1-x^2},-\frac{t}{\sqrt{1-2t}}\leq x\leq \frac{t}{\sqrt{1-2t}}$.

We claim that  the intersection of $T_1(\alpha_{4,t})$ with  $T_1(\mathcal{S}_{3,t})$ is empty.
Substituting (\ref{eq:Talpha4}) to the equation of the bisector $T_1(\mathcal{B}_{3,t})$
$$|\langle V_t, T_1(q_0) \rangle|=|\langle V_t, T_1g_2^{2}g_1(q_0)\rangle|.$$
We get the equation

$$k_{3,1}x+k_{3,2}y+k_{3,0}=0,$$
where
\begin{equation*}
	\begin{aligned} 
		&k_{3,1}=-\frac{24c(t)(t-1/3)\left(\sqrt{2-6t}(a^2+a/2-1/4)-3/2t(t-1/2)\right)}{(2t-1)^2}, \\
		&k_{3,2}=\frac{24c(t)(t-1/3)\left((t/2-1/4)\sqrt{t^2+12t-1}+\sqrt{2}t\sqrt{(3t-1)(a^2+2t-1)}\right)}{(2t-1)^2},\\
		&k_{3,0}=\frac{24(t-1/3)(3t^2-7/2t+3/4)}{(2t-1)^2}.
	\end{aligned}
\end{equation*}
By using some computer algebra software, we find that the minimum of the  expression 
$$\frac{k_{3,0}^2}{k_{3,1}^2+k_{3,2}^2}$$
is given approximately by 6.5907 for  $t\in [3/8,\sqrt{2}-1]$.
So the family of lines does not intersect the circle $x^2+y^2=1$. 
 Thus $\alpha_{4,t}$ does not intersect with the spinal sphere  $\mathcal{S}_{3,t}$. With the same argument, one can also prove that $\alpha_{4,t}$ does not intersect with $\mathcal{S}_{2,t}$, $\mathcal{S}_{6,t}$, $\mathcal{S}_{7,t}$.

Substituting (\ref{eq:Talpha4}) to the equation of the bisector $T_1(\mathcal{B}_{1,t})$
$$|\langle V_t, T_1q_0 \rangle|=|\langle V_t, T_1g_1(q_0)\rangle|.$$
We get the equation
\begin{equation}\label{eq:B1-alpha4}
k_{1,1}x+k_{1,2}y+k_{1,0}=0,
\end{equation}

where
\begin{equation*}
	\begin{aligned} 
		&k_{1,1}=\frac{(6t-2)c(t)\left(t-\sqrt{2-6t}\right)}{2t-1}, \\
		&k_{1,2}=\frac{(6t-2)c(t)\sqrt{t^2+12t-1}}{2t-1},\\
		&k_{1,0}=\frac{(6t-2)(8t-3)}{2t-1}.
	\end{aligned}
\end{equation*}

The intersection point corresponds to the solution

\begin{equation*}
	\begin{aligned} 
		y&=\frac{-k_{1,0}k_{1,2}-\sqrt{k_{1,1}^4+k_{1,1}^2k_{1,2}^2-k_{1,0}^2k_{1,1}^2}}{k_{1,1}^2+k_{1,2}^2}, \\
		x&=-\frac{k_{1,2}y+k_{1,0}}{k_{1,1}}.
	\end{aligned}
\end{equation*}
Similarly, it can be showed that the  $\C$-arc  $\alpha_{4,t}$ has only one intersection with the spinal spheres $\mathcal{S}_{4,t}$, $\mathcal{S}_{5,t}$ and $\mathcal{S}_{8,t}$.

In the last step,  we show that  $\alpha_{4,t}$ can not pass through the intersection of the spinal spheres $\mathcal{S}_{1,t}$ and $\mathcal{S}_{5,t}$.

Substituting (\ref{eq:Talpha4}) to the equation of the bisector $T_1(\mathcal{B}_{5,t})$
$$|\langle V_t, T_1q_0 \rangle|=|\langle V_t, T_1g_1(q_0)\rangle|.$$
We get the equation
\begin{equation}\label{eq:B5-alpha4}
	k_{5,1}x+k_{5,2}y+k_{5,0}=0,
\end{equation}

where
\begin{equation*}
	\begin{aligned} 
		&k_{5,1}=-\frac{6(t-1/3)(t-1/4)c(t)\left(t-\sqrt{2-6t}\right)}{(t-1/2)^2}, \\
		&k_{5,2}=-\frac{6(t-1/3)(t-1/4)c(t)\sqrt{t^2+12t-1}}{(t-1/2)^2},\\
		&k_{5,0}=\frac{6(t-1/3)(2t-3/4)}{(t-1/2)^2}.
	\end{aligned}
\end{equation*}

Then we have
$$k_{1,1}=k_{5,1}\frac{1-2t}{4t-1},\quad k_{1,2}=k_{5,2}\frac{1-2t}{4t-1},\quad k_{1,0}=k_{5,0}(1-2t).$$

So the equations (\ref{eq:B1-alpha4}) and (\ref{eq:B5-alpha4}) has no common solution.  With the same argument, we can prove that  $\alpha_{4,t}$ can not pass through the intersection of the spinal spheres $\mathcal{S}_{4,t}$ and $\mathcal{S}_{8,t}$. 

\end{proof}

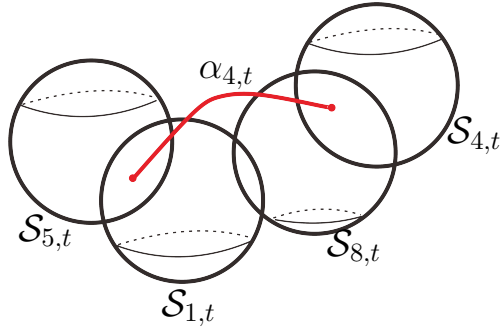
\begin{figure}
	\begin{center}
		\begin{tikzpicture}
			\node at (0,0) {\includegraphics[width=6cm]{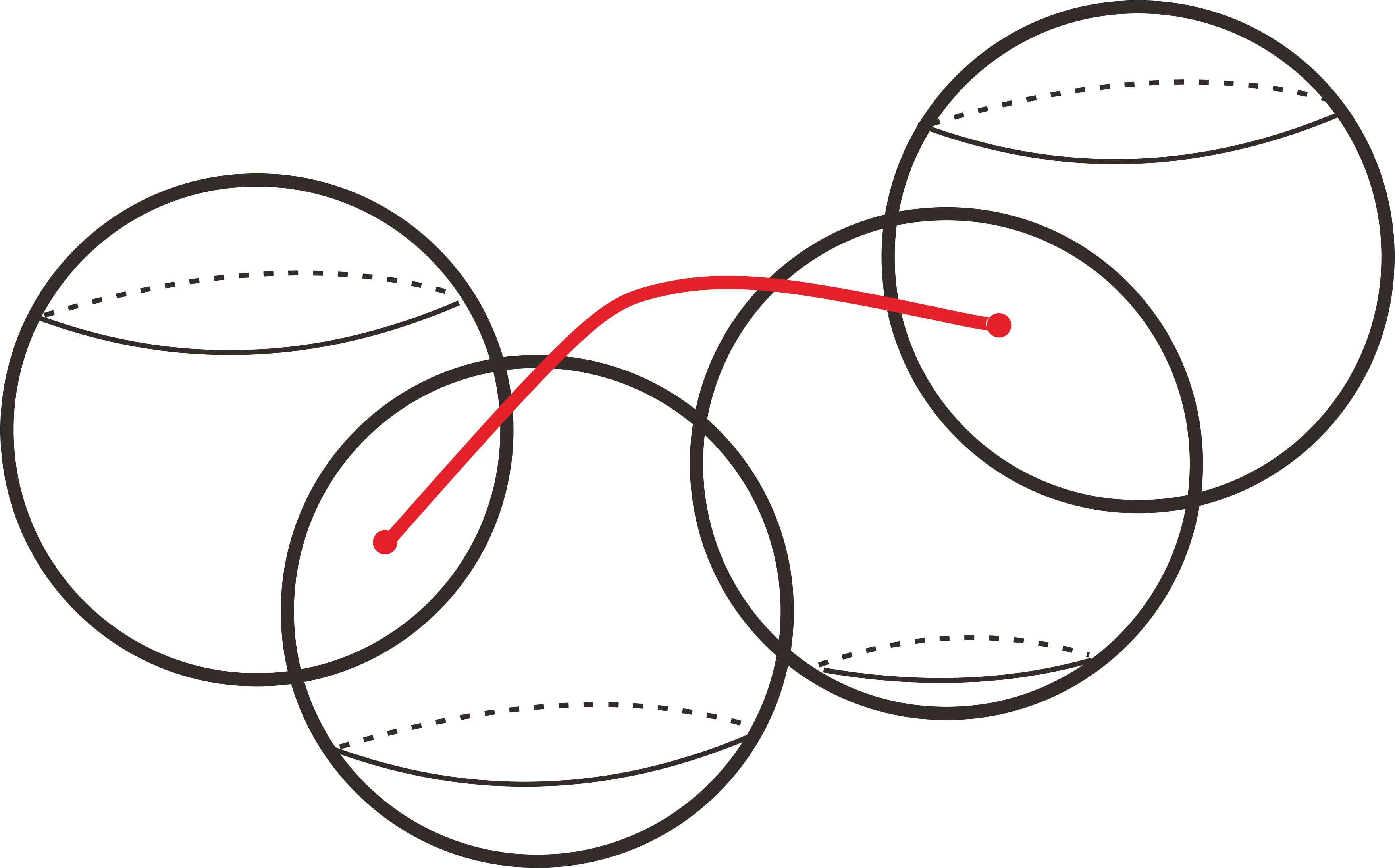}};
	\node at (-0.1,0.9){\large $\alpha_{4,t}$};
	\node at (-0.6,-2.2){\large $\mathcal{S}_{1,t}$};
	\node at(-2.5,-1.2){\large$\mathcal{S}_{5,t}$};	
	\node at(1.6,-1.4){\large$\mathcal{S}_{8,t}$};	
	\node at(3.2,0.1){\large$\mathcal{S}_{4,t}$};					
		\end{tikzpicture}
	\end{center}
	\caption{ The configuration of arc $\alpha_{4,t}$ and the spinal spheres $\mathcal{S}_{1,t}$, $\mathcal{S}_{4,t}$, $\mathcal{S}_{5,t}$, $\mathcal{S}_{8,t}$.}
	\label{fig:alpha4-arc}
\end{figure}


\begin{figure}
	\begin{center}
		\begin{tikzpicture}
			\node at (0,0) {\includegraphics[width=12cm,height=4cm]{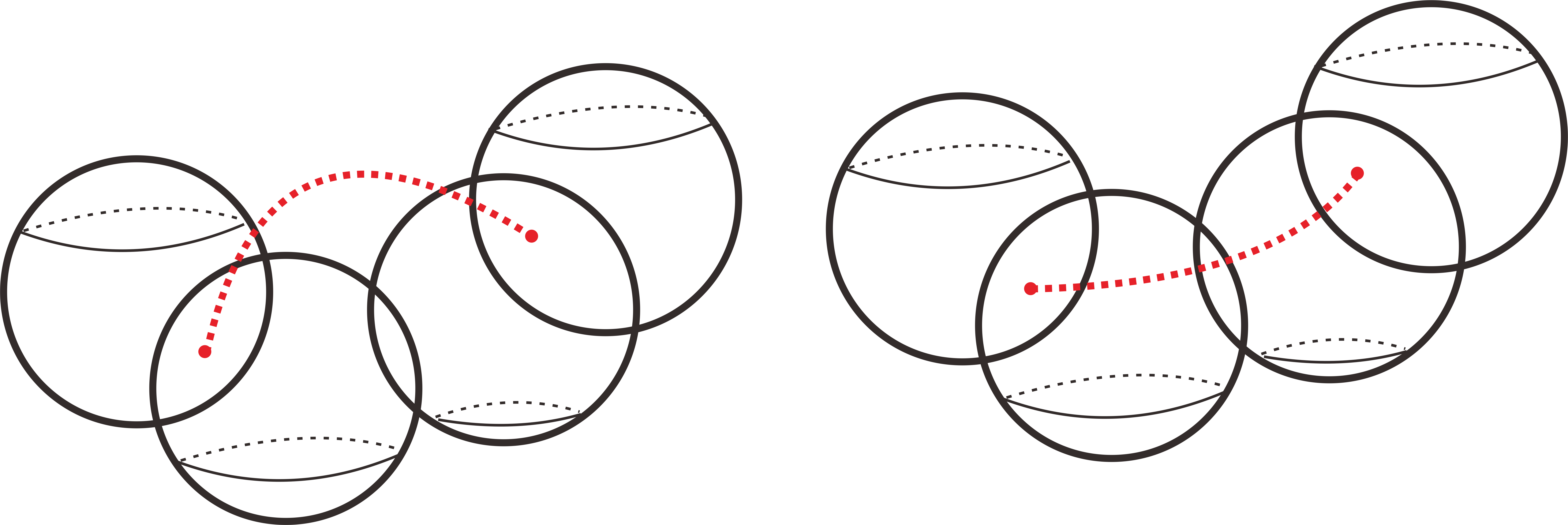}};
			\node at (-3.1,0.9){\large $\alpha_{4,t}$};
			\node at (-5.6,-1.4){\large $\mathcal{S}_{5,t}$};
			\node at(-3.7,-2.3){\large$\mathcal{S}_{1,t}$};	
			\node at(-1.9,-1.7){\large$\mathcal{S}_{8,t}$};	
			\node at(-1.4,1.8){\large$\mathcal{S}_{4,t}$};
			
		\node at (2.7,0){\large $\alpha_{4,t}$};
		\node at (0.9,-1){\large $\mathcal{S}_{5,t}$};
		\node at(4.3,-1.2){\large$\mathcal{S}_{8,t}$};	
		\node at(2.5,-1.8){\large$\mathcal{S}_{1,t}$};	
		\node at(5.1,2.2){\large$\mathcal{S}_{4,t}$};		
								
		\end{tikzpicture}
	\end{center}
	\caption{ The impossible configurations of arc $\alpha_{4,t}$ and the spinal spheres $\mathcal{S}_{1,t}$, $\mathcal{S}_{4,t}$, $\mathcal{S}_{5,t}$, $\mathcal{S}_{8,t}$ during the deformation. Compare this with Figure \ref{fig:alpha4-arc}.}
	\label{fig:alpha4-arc2}
\end{figure}
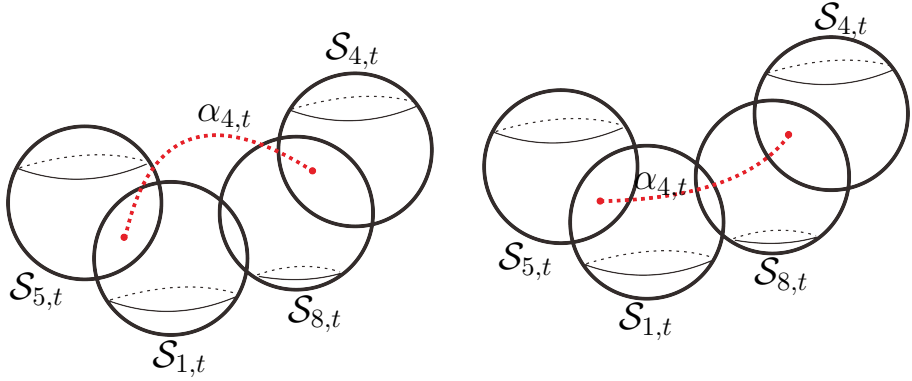

\begin{prop} \label{prop:disjoint} The eight cutting disks are disjointed during the deformation.
\end{prop}
\begin{proof}

First, we note that each pair of the eight $\C$-circles	containing the  $\C$-arcs is not linked for $t\in [3/8,2/5]$. Therefore, both these eight discs and their corresponding cutting discs do not intersect.
These observations suggest dividing the analysis into two cases.

{\bf Case 1}: $t\in [3/8,2/5)$. Let $v_{1,t}$,  $v_{2,t}$ be the polar	vectors of the $\C$-circles containing the $\C$-arc $\alpha_{1,t}$ and $\C$-arc $\beta_{1,t}$. Then 
\begin{equation*}
v_{1,t}=\left[
\begin{array}{c}
	-1 \\
	-\frac{\sqrt{2}a(t)(-t+b(t)i)}{2\sqrt{2t-1}} \\
	1\\
\end{array}
\right]\quad {\rm and}\quad v_{2,t}=\left[
	\begin{array}{c}
		-\frac{t-a(t)-1+b(t)i}{4\sqrt{1-2t}}\\
	\frac{\sqrt{2}a(t)}{4\sqrt{1-2t}}\\
	\frac{t+a(t)-1+b(t)i}{4\sqrt{1-2t}} \\
	\end{array}
	\right].
\end{equation*}
	
By the non-linking condition (\ref{eq:linked-circle}), we have
$$L(v_{1,t},v_{2,t})=\frac{2(1-3t)}{2t-1}.$$
It is easy to see that  the $\C$-circles containing the $\C$-arc $\alpha_{1,t}$ and $\C$-arc $\beta_{1,t}$ can not be linked.

A simple calculation yield
$$L\left(v_{1,t},g_2(v_{1,t})\right)=-\frac{2(15t^2-11t+2)}{(2t-1)^2}.$$
So  the $\C$-circles containing the $\C$-arcs $\alpha_{1,t}$ and $\C$-arc $\alpha_{2,t}$ can not be linked.

 Similar calculations will allow us to see that each pair of the eight $\C$-circles containing the eight $\C$-arcs  can not be linked for $3/8\leq t<2/5$.
 
 {\bf Case 2}: $t\in[2/5,\sqrt{2}-1]$. As an example, we only show that the cutting disks corresponds to the  $\C$-arcs $\alpha_{1,t}$ and $\C$-arc $\alpha_{2,t}$ are disjointed. 
 
 In this case,  the $\C$-circles containing  $\C$-arcs $\alpha_{1,t}$ and $\C$-arc $\alpha_{2,t}$ are linked.

From the polar vector $v_{1,t}$, we see that the contact plane containing the  $\C$-arcs $\hat{\alpha_1}$ 
 based at the point with Heisenberg coordinate $[x_1,y_1,t_1]$, where
  \begin{equation*}
 	x_1=\frac{t\sqrt{2}a(t)}{2\sqrt{2t-1}},\, y_1=-\frac{\sqrt{2}a(t)b(t)}{2\sqrt{2t-1}},\, t_1=0.
 \end{equation*}
 The projection of the   $\C$-circle to $\C$-plane is  Euclidean circle
 center at $(x_1,y_1)$ with radius $r_1=\sqrt{\frac{6-16t}{2t-1}}$.
 \begin{equation*}
 	\left[\frac{t\sqrt{2}a(t)}{2\sqrt{2t-1}},-\frac{\sqrt{2}a(t)b(t)}{2\sqrt{2t-1}},0\right],\quad \left[0,0,\frac{-4a(t)b(t)}{2ta(t)+4t-1}\right]
 \end{equation*}	
 respectively.

  After normalization, the polar vector
 $g_2(v_{1,t})$ of the $\C$-arc $\alpha_{2,t}$ is given by 
 \begin{equation*}
 	\left[
 	\begin{array}{c}
 		\frac{8t-3-2a(t) b(t)i}{2t a(t)+4t-1}\\
 		0\\
 		1\\
 	\end{array}
 	\right].
 \end{equation*}
  Then the contact plane containing the  $\C$-arc $\hat{\alpha_{2,t}}$ and $\hat{\alpha_{2,t}}$  based at the points with Heisenberg coordinates  $[x_2,y_2,t_2]$
  where
 \begin{equation*}
 	x_2=0,\, y_2=0,\, t_2=\frac{-4a(t) b(t)}{2t a(t)+4t-1}.
 \end{equation*}
 The projection of the  $\C$-circle containing $\alpha_{2,t}$ to $\C$-plane is  Euclidean circle
 center at $(x_2,y_2)$ with radius $r_2=\sqrt{\frac{16t-6}{2t a(t)+4t-1}}$.
 
 Define
 \begin{equation*}
 	k_1=-\frac{b(t)}{t},\quad k_2=-\frac{\sqrt{2}(2t-1)b(t)}{t(2t a(t)+4t-1)}.
 \end{equation*}

  The intersection of these  contact planes is an affine line given by 
 \begin{equation*}
 	\{[x,k_1x+k_2,t_2] | x\in \R \}.
 \end{equation*}
 
 By studying the intersection of the affine line with the $\C$-circles, we get  that the intersection of  these two affine disks is an affine segment given by 
 \begin{equation*}
 	\{[x,0,0] | \iota_1 \leq x\leq \iota_2  \},
 \end{equation*}
 where
 
 \begin{equation*}
 	\begin{aligned}
 		&\iota_1=\frac{\sqrt{k_1^2r_2^2-k_2^2+r_2^2}-k_1k_2}{1+k_1^2}, \\
 		&\iota_2=\frac{k_1y_1+x_1-k_1k_2+\sqrt{(k_1y_1+x_1-k_1k_2)^2-(k_2^2-2y_1k_2+2)(1+k_1^2)}}{1+k_1^2}. \\ 
 	\end{aligned}
 \end{equation*}

Define
 \begin{equation*}
 	v_s=\left[
 	\begin{array}{c}
 		\frac{-\left(\iota_1s+(1-s)\iota_2\right)^2}{2} \\
 	\iota_1s+(1-s)\iota_2 \\
 		1 \\
 	\end{array}
 	\right],
 \end{equation*}
where $s\in [0,1]$.
 By using some computer algebra software, we find that the minimum of the  expression 
 $$|\langle  v_s, q_0 \rangle|^2-|\langle  v_s, g_2g_1(q_0)\rangle|^2$$
 is  given approximately by 0.3616753 for $s\in [0,1]$ and $t\in [2/5,\sqrt{2}-1]$. We omit writing the explicit expression, because it is a bit too complicated to fit on paper. This means that the affine segment lies inside the spinal sphere  $\mathcal{S}_{2,t}$.  So the intersection of the cut disks corresponding to the  arcs $\hat{\alpha_{1,t}}$ and $\hat{\alpha_{2,t}}$ is empty, see Figure \ref{fig:disjoint-disk}.

\begin{figure}
	\begin{center}
		\begin{tikzpicture}
			\node at (0,0) {\includegraphics[width=6cm]{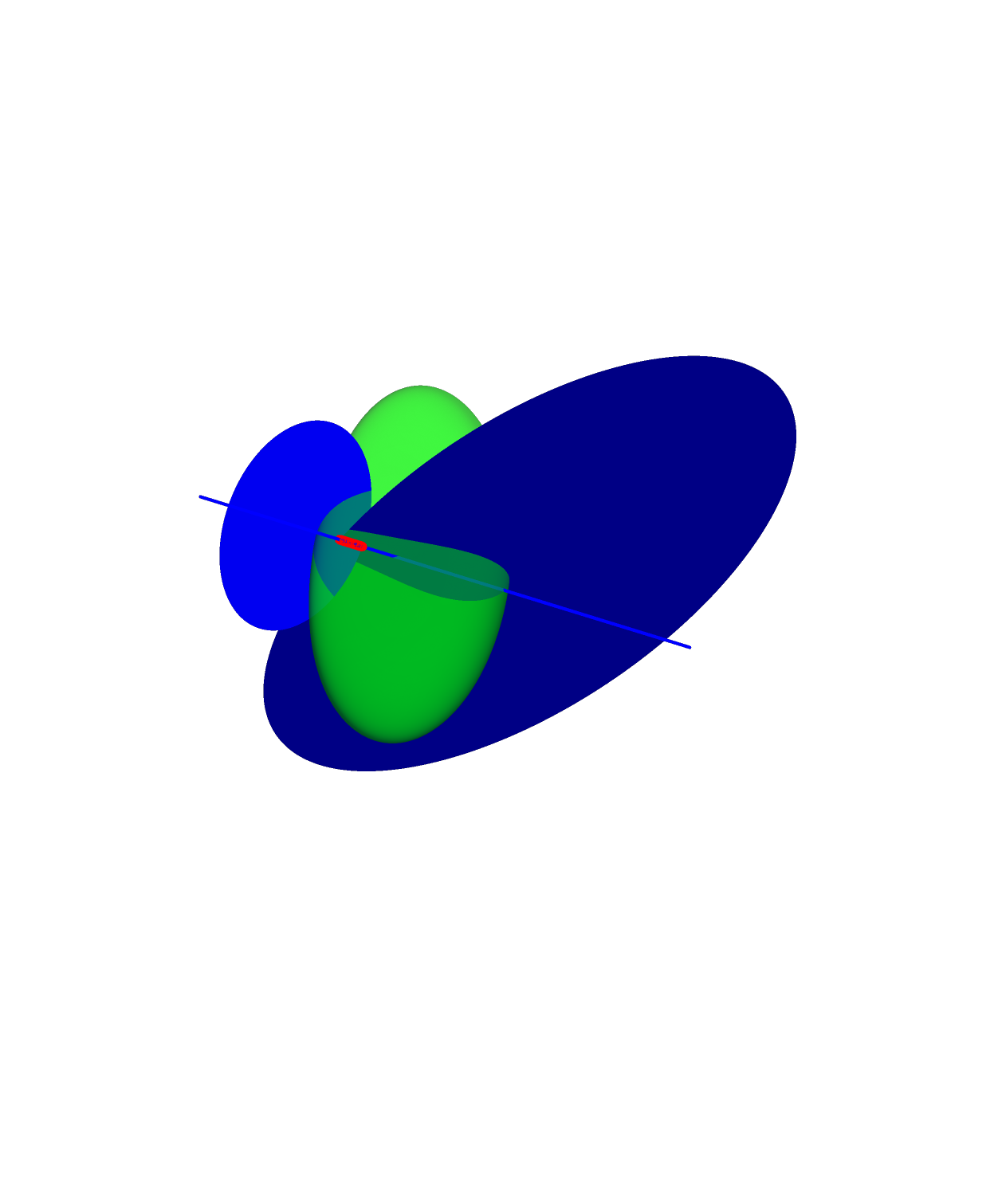}};
			\node at (-0.5,1.6){\large $\mathcal{S}_{2,t}$};					
		\end{tikzpicture}
	\end{center}
	\caption{Two blue affine disks contain the $\C$-arcs $\hat{\alpha_{1,t}}$ and $\hat{\alpha_{2,t}}$, respectively. The intersection of two blue affine disks is the red affine segment, which lies in spinal sphere $\mathcal{S}_{2,t}$.}
\label{fig:disjoint-disk}
\end{figure}

\end{proof}

\textbf{The proof of Theorem \ref{thm:figure8}. } Similar to  Proposition \ref{prop:CrownRfuchsian}, for $t \in (3/8,  \sqrt{2}-1]$, the quotient space of  $$\partial_{\infty}D_{t}-(\cup^4_{i=1}(\hat{\alpha_{i,t}}\cup \hat{\beta_{i,t}}))$$ by the natural side-pairings  on $$ \cup^8_{j=1} \mathcal{A}_{i,t}\cup(\cup^4_{i=1}(\hat{\alpha_{i,t}}\cup \hat{\beta_{i,t}}))$$ is homeomorphic to the quotient space  $\Omega_{\Gamma_{t},g_1}/\Gamma_{t}$. 
By the geometric stability in the deformation,  that is,  Propositions  \ref{prop:arcends}
and  \ref{prop:disjoint}, the topology and combinatoris of $\partial_{\infty}D_{t}-(\cup^4_{i=1}(\hat{\alpha_{i,t}}\cup \hat{\beta_{i,t}}))$ does not change in the deformation, and the side-pairing pattern also does not change. So 
the quotient space of $\partial_{\infty}D_{t}-(\cup^4_{i=1}(\hat{\alpha_{i,t}}\cup \hat{\beta_{i,t}}))$ is homeomorphic to the quotient space of $\partial_{\infty}D_{t_0}-(\cup^4_{i=1}(\hat{\alpha_{i}}\cup \hat{\beta_{i}}))$ whenever  $t \in (3/8,  \sqrt{2}-1]$. By  Proposition \ref{prop:CrownRfuchsian}, the quotient space is the   figure-eight knot complement. This ends the proof of Theorem \ref{thm:figure8}.

\end{document}